\documentclass{amsart}
\usepackage{latexsym,amsxtra,amscd,ifthen}
\usepackage{amsfonts}
\usepackage{verbatim}
\usepackage{amsmath}
\usepackage{amsthm}
\usepackage{amssymb}
\usepackage{color}
\numberwithin{equation}{section}

\theoremstyle{plain}
\newtheorem{theorem}{Theorem}[section]
\newtheorem{lemma}[theorem]{Lemma}
\newtheorem{proposition}[theorem]{Proposition}
\newtheorem{hypothesis}[theorem]{Hypotheses}
\newtheorem{corollary}[theorem]{Corollary}

\newcommand{\inth}{\textstyle \int}
\theoremstyle{definition}
\newtheorem{definition}[theorem]{Definition}
\newtheorem{example}[theorem]{Example}

\newtheorem{remark}[theorem]{Remark}
\newtheorem{question}[theorem]{Question}
\newtheorem{notation}[theorem]{Notation}

\makeatletter              
\let\c@equation\c@theorem  
\makeatother

\newcommand{\bfl}{\mathfrak l}

\DeclareMathOperator{\hdet}{hdet}

\DeclareMathOperator{\Ext}{Ext} \DeclareMathOperator{\Tor}{Tor}

 \DeclareMathOperator{\ann}{ann}

\DeclareMathOperator{\Aut}{Aut}
\DeclareMathOperator{\rk}{rk}

\DeclareMathOperator{\GKdim}{GKdim}
\DeclareMathOperator{\End}{End} 
 \DeclareMathOperator{\Hom}{Hom}

\DeclareMathOperator{\GrMod}{{\sf GrMod}}

\newcommand{\fm}{\mathfrak{m}}
\newcommand{\p}{{\sf p}}
\newcommand{\HH}{{\text{H}}^d_{\fm}}

\newcommand{\be}{\begin{enumerate}}
\newcommand{\ee}{\end{enumerate}}
\newcommand{\bq}{\begin{eqnarray*}}
\newcommand{\eq}{\end{eqnarray*}}
\newcommand{\bqn}{\begin{eqnarray}}
\newcommand{\eqn}{\end{eqnarray}}

\begin{document}
\title[Reflection Hopf algebras]
{The Jacobian, reflection arrangement and 
discriminant for reflection Hopf algebras}

\author{E. Kirkman and J.J. Zhang}

\address{Kirkman: Department of Mathematics,
P. O. Box 7388, Wake Forest University, Winston-Salem, NC 27109}

\email{kirkman@wfu.edu}

\address{Zhang: Department of Mathematics, Box 354350,
University of Washington, Seattle, Washington 98195, USA}

\email{zhang@math.washington.edu}

\begin{abstract}
We study finite dimensional semisimple Hopf algebra actions on 
noetherian connected graded Artin-Schelter regular algebras 
and introduce definitions of the Jacobian, the reflection 
arrangement and the discriminant in a noncommutative setting.
\end{abstract}





\maketitle


\setcounter{section}{-1}

\section{Introduction}
\label{xxsec0}

The Shephard-Todd-Chevalley Theorem states that if $G$ is a 
finite group acting linearly and faithfully on the commutative 
polynomial ring $\Bbbk[x_1,\cdots, x_n]$, where the 
characteristic of the base field $\Bbbk$ is zero, the fixed 
subring ${\Bbbk}[x_1,\cdots, x_n]^G$ is isomorphic to 
$\Bbbk[x_1,\cdots, x_n]$ if and only if $G$ is generated by 
pseudo-reflections of the space $V:=\bigoplus_{i=1}^n {\Bbbk} 
x_i^{\ast}$. Such a group $G$ is called a {\it reflection group}. 
Note that ${\Bbbk}[x_1,\cdots, x_n]$ is the ring of regular 
functions on $V$. This paper is part of a project to extend
properties of the action of reflection groups on commutative 
polynomial algebras to a noncommutative setting.

In the noncommutative setting we consider here, the commutative 
polynomial ring ${\Bbbk}[x_1,\cdots, x_n]$ is replaced by an 
Artin-Schelter regular $\Bbbk$-algebra, denoted by $A$, and the 
group $G$ (or the group ring $\Bbbk G$) is replaced by a (finite 
dimensional) semisimple Hopf $\Bbbk$-algebra, denoted by $H$. We 
say $H$ is a {\it reflection Hopf algebra} or {\it reflection 
quantum group} if the fixed subring $A^H$ \eqref{E0.1.2} is again 
Artin-Schelter regular \cite[Definition 3.2]{KKZ4}. The first 
example of a noncommutative and noncocommutative reflection Hopf 
algebra (the Kac-Palyutkin algebra \cite{KP}  acting on 
$\Bbbk_{i}[u,v]$ where $i^2=-1$) was given in 
\cite[Example 7.4]{KKZ3}. A systematic study of dual reflection 
groups (where $H = (\Bbbk G)^\ast$) was begun in \cite{KKZ4}. 
This noncommutative (and noncocommutative) context for 
noncommutative invariant theory has proved fruitful, and results 
include:

\begin{enumerate}
\item[(a)]
The rigidity of (noetherian) Artin-Schelter regular algebras under finite group or
semisimple Hopf algebra actions \cite{AP, CKZ1, KKZ1, KKZ4}.
\item[(b)]
The homological determinant and Watanabe's theorem. The homological 
determinant of a group action on Artin-Schelter regular algebras 
was introduced in \cite{JZ}, and that of a Hopf action in \cite{KKZ3}.
\item[(c)]
The Nakayama automorphism and twisted (skew) Calabi-Yau property
\cite{CWZ, KKZ4, RRZ2, RRZ3}.
\item[(d)]
The pertinency and radical ideal associated to Hopf actions on 
Artin-Schelter regular algebras, Auslander's theorem, the McKay 
correspondence, and noncommutative resolutions 
\cite{BHZ1, BHZ2, CKWZ2, CKWZ3, CKZ2, GKMW, QWZ}.
\end{enumerate}

A survey on noncommutative invariant theory in this context is given 
in \cite{Ki}. 

An important  topic in classical invariant theory is the 
arrangements of hyperplanes associated to reflection groups \cite{OT}.
It is related to combinatorics, algebra, geometry, representation
theory, complex analysis and other fields. 

In this paper we investigate the possibility of defining a noncommutative 
version of a hyperplane arrangement. Some fundamental work of 
Steinberg \cite{Ste}, Stanley \cite{Sta}, Terao \cite{Te}, 
Hartmann-Shepler \cite{HS}, Orlik-Terao \cite{OT} and many others 
offered an algebraic approach that can be adapted to the 
noncommutative case. In particular, we will introduce 
a few concepts that characterize significant structures of  the
actions of reflection Hopf algebras on Artin-Schelter regular algebras. 

Throughout the rest of this paper, let $\Bbbk$ be a base field that is 
algebraically closed, and all vector spaces, (co)algebras, Hopf 
algebras, and morphisms are over $\Bbbk$. In general we do not need to 
assume that the characteristic of $\Bbbk$ is zero. However, in several 
places where we use results from other papers (e.g. \cite{KKZ2, KKZ4}), 
we add the characteristic zero hypothesis because those results were 
proved under that extra hypothesis. Let $H$ denote a semisimple (hence 
finite dimensional) Hopf algebra, and let $K$ be the $\Bbbk$-linear 
Hopf dual $H^*$ of $H$. Throughout we use standard notation 
(see e.g. \cite{Mo1}) for a Hopf algebra $H(\Delta, \epsilon, S)$. 
It is well-known that a left $H$-action on an algebra $A$ is 
equivalent to a right $K$-coaction on $A$, and we will use this 
fact freely.  Let $\GKdim A$ denote the Gelfand-Kirillov dimension 
of the algebra $A$ \cite{KL}. Let $\Bbbk^\times$ be the set of 
invertible elements in $\Bbbk$. If $f, g\in A$ and $f = cg$ for 
some $c\in \Bbbk^{\times}$, then we write $f =_{\Bbbk^{\times}} g$.

\begin{hypothesis}
\label{xxhyp0.1} Assume the following hypotheses:
\begin{enumerate}
\item[(a)]
$A$ is a noetherian connected graded Artin-Schelter regular algebra 
that is a domain, see Definition \ref{xxdef1.1}, 
\item[(b)]
$H$ is a semisimple Hopf algebra,
\item[(c)]
$H$ acts on $A$ inner-faithfully \cite[Definition 1.5]{CKWZ1} 
and homogeneously so that $A$ is a left $H$-module algebra,
\item[(d)]
$H$ acts on $A$ as a reflection Hopf algebra in the sense of 
\cite[Definition 3.2]{KKZ4}, or equivalently, of Definition \ref{xxdef1.4}.
\end{enumerate}
\end{hypothesis}

Let $G(K)$ be the group of grouplike elements in $K:=H^{\ast}$. 
For each $g\in G(K)$, define
\begin{equation}
\label{E0.1.1}\tag{E0.1.1}A_{g}:=\{a\in A\mid \rho(a)=a \otimes g\},
\end{equation}
where $\rho: A \rightarrow A \otimes K$ is the corresponding right 
coaction of $K$ on $A$. The fixed subring of the $H$-action on $A$
is defined to be
\begin{equation}
\label{E0.1.2}\tag{E0.1.2}
A^H:=\{ a\in A\mid h\cdot a=\epsilon(h) a\; \forall \; h\in H\}.
\end{equation}
We refer to \cite[Section 3]{KKZ3} for the definition of the 
{\it homological determinant} of the $H$-action on $A$. Let 
$\hdet: H\to \Bbbk$ be the homological determinant of the 
$H$-action on $A$. Then $\hdet$, considered as an element in $K$, 
is a grouplike element. By \cite[Theorem 0.6]{CKWZ1}, $\hdet$ is 
nontrivial (unless $A=A^H$) when $H$ is a reflection Hopf algebra.
Since $\hdet$ is an element in $G(K)$, both $A_{\hdet}$ and 
$A_{\hdet^{-1}}$ are defined by \eqref{E0.1.1}.

\begin{theorem}[Corollary \ref{xxcor2.5}(1) and Theorem \ref{xxthm3.8}(1)]
\label{xxthm0.2} Assume Hypotheses \ref{xxhyp0.1}. Let $R$ be the 
fixed subring $A^H$.
\begin{enumerate}
\item[(1)]
There is a nonzero element \; ${\mathsf j}_{A,H}\in A$, unique up to a nonzero
scalar, such that $A_{\hdet^{-1}}$ is a free $R$-module of rank one on 
both sides generated by ${\mathsf j}_{A,H}$.
\item[(2)]
There is a nonzero element \; ${\mathsf a}_{A,H}\in A$, unique up to a nonzero
scalar, such that $A_{\hdet}$ is a free $R$-module of rank one on 
both sides generated by ${\mathsf a}_{A,H}$.
\item[(3)]
The products ${\mathsf j}_{A,H}{\mathsf a}_{A,H}$ and 
${\mathsf a}_{A,H}{\mathsf j}_{A,H}$ in $A$ are elements of $R$ 
that are either equal, or they differ only by 
a nonzero scalar in $\Bbbk$, or equivalently, ${\mathsf j}_{A,H}{\mathsf a}_{A,H}
=_{\Bbbk^{\times}} {\mathsf a}_{A,H}{\mathsf j}_{A,H}$.
\end{enumerate}
\end{theorem}

The above theorem allows us to define the following fundamental concepts.

\begin{definition}
\label{xxdef0.3}
Assume Hypotheses \ref{xxhyp0.1}.
\begin{enumerate}
\item[(1)]
The element ${\mathsf j}_{A,H}$ in Theorem \ref{xxthm0.2}(1) is called 
the {\it Jacobian} of the $H$-action on $A$.
\item[(2)] 
The element ${\mathsf a}_{A,H}$ in Theorem \ref{xxthm0.2}(2) is called 
the {\it reflection arrangement} of the $H$-action on $A$.
\item[(3)] 
The element ${\mathsf j}_{A,H}{\mathsf a}_{A,H}$, or equivalently,
${\mathsf a}_{A,H}{\mathsf j}_{A,H}$, 
in Theorem \ref{xxthm0.2}(3) is called the {\it discriminant} of the 
$H$-action on $A$, and denoted by $\delta_{A,H}$. 
\end{enumerate}
The above concepts are well-defined up to a nonzero scalar in $\Bbbk$, 
and under some hypotheses we show they exist more generally.
\end{definition}

In the  classical (commutative) setting, when $G$ is a reflection 
group acting on a vector space $V$ over the field of complex numbers 
${\mathbb C}$, the Jacobian (respectively, the reflection arrangement, 
the discriminant) in Definition \ref{xxdef0.3} is essentially equivalent 
to the classical Jacobian determinant of the basic invariants of $G$ in 
the commutative polynomial ring ${\mathbb C}[V^{\ast}]:={\mathcal O}(V)$
(respectively, the reflection arrangement, the discriminant of the 
$G$-action). When we let $A={\mathbb C}[V^{\ast}]$ and $H={\mathbb C} G$ 
in Hypotheses \ref{xxhyp0.1}, a well-known result of Steinberg \cite{Ste} 
states that 
\begin{equation}
\label{E0.3.1}\tag{E0.3.1}
{\mathsf j}_{A,H}=_{{\mathbb C}^{\times}}\prod_{s=1}^{v} f_s^{e_s-1}
\end{equation}
where $\{f_s\}_{s=1}^v$ is the complete list of the linear equations of the 
reflecting hyperplanes of $G$, and each $e_s$ is the exponent of the 
pointwise stabilizer subgroup that consists of pseudo-reflections 
in $G$ associated to the corresponding reflecting hyperplane. After 
we identify each hyperplane in $V$ with its linear form in 
$V^{\ast}$, the set of reflecting hyperplanes is uniquely 
determined by the following equation \cite[Examples 6.39 and 6.40]{OT}
(where $\det$ and $\det^{-1}$ are switched due to different convention 
used in the book \cite{OT})
\begin{equation}
\label{E0.3.2}\tag{E0.3.2}
{\mathsf a}_{A,H}=_{{\mathbb C}^{\times}}\prod_{s=1}^{v} f_s,
\end{equation}
which suggests calling ${\mathsf a}_{A,H}$ in Definition \ref{xxdef0.3} 
the reflection arrangement of the $H$-action on $A$. In this paper, we 
can prove only the following weaker version of Steinberg's theorem 
\cite{Ste, HS} in the noncommutative setting [Theorem \ref{xxthm0.5}].

\begin{hypothesis}
\label{xxhyp0.4} Assume the following hypotheses:
\begin{enumerate}
\item[(1)]
Assume Hypotheses \ref{xxhyp0.1}.
\item[(2)]
${\rm{char}}\; \Bbbk=0$.
\item[(3)]
$H$ is commutative, or equivalently, $H=(\Bbbk G)^{\ast}$ for 
some finite group $G$.
\item[(4)]
$A$ is generated in degree one.
\end{enumerate}
\end{hypothesis}

\begin{theorem}[Theorem \ref{xxthm2.12}(2)]
\label{xxthm0.5} 
Assume Hypotheses \ref{xxhyp0.4}. Then the following hold. 
\begin{enumerate}
\item[(1)] 
${\mathsf j}_{A,H}$ is a product of elements of degree one. 
\item[(2)] 
${\mathsf a}_{A,H}$ is a product of elements of degree one. 
\end{enumerate}
\end{theorem}

When $A$ is noncommutative, it is usually not a unique factorization 
domain. Then the decompositions of ${\mathsf j}_{A,H}$ and
${\mathsf a}_{A,H}$ into products of linear forms in Theorem 
\ref{xxthm0.5}, formulas like \eqref{E0.3.1}-\eqref{E0.3.2}, are not 
unique, see Examples \ref{xxex2.2}(2) and \ref{xxex4.2}. Therefore 
it is difficult to imagine and define individual reflecting 
hyperplane at this point, though, in some special cases, there are 
natural candidates for such hyperplanes, see \eqref{E2.2.2}. 
We have some general results as follows.

\begin{theorem}[Theorem \ref{xxthm3.8}(2)]
\label{xxthm0.6} 
Assume Hypotheses \ref{xxhyp0.1}. Then ${\mathsf a}_{A,H}$ 
divides ${\mathsf j}_{A,H}$ from the left and the right. 
\end{theorem}

In the classical setting, when $H=\Bbbk G$, for $G$  a reflection group 
acting on a vector space $V$, then $\delta_{A,H}$ agrees with 
the classical definition of discriminant of the $G$-action
\cite[Definition 6.44]{OT}. When $R$ is central in $A$ and $H$ is a 
dual reflection group, then $\delta_{A,H}$ is closely related to 
the noncommutative discriminant $dis(A/R)$ studied in 
\cite{BZ, CPWZ1, CPWZ2}.

\begin{theorem}[Theorem \ref{xxthm3.10}(2)]
\label{xxthm0.7} 
Assume Hypotheses \ref{xxhyp0.4}. Suppose that $R:=A^H$ \eqref{E0.1.2} 
is central in $A$. Then $\delta_{A,H}$ and $dis(A/R)$ have the same 
prime radical.
\end{theorem}

We refer to \cite{YZ, Zh} for the definition of Auslander regularity 
and \cite[Definition 0.1]{JZ} for the definition of Artin-Schelter 
Cohen-Macaulay used in the next theorem and its proof. By Theorem
\ref{xxthm2.4}, the Jacobian ${\mathsf j}_{A,H}$ can be defined in a more
general setting, which is used in the next theorem.

\begin{theorem}[Theorem \ref{xxthm3.9}]
\label{xxthm0.8} 
Assume Hypotheses \ref{xxhyp0.1}. Suppose $A^H$ is Auslander regular.
Then $A_G:=\bigoplus_{g\in G(K)} A_g$ is Artin-Schelter Gorenstein 
and 
$${\mathsf j}_{A,H}=_{\Bbbk^{\times}}{\mathsf j}_{A_G, (\Bbbk G)^{\ast}}.$$
\end{theorem}

The theorem above leads to the following question.

\begin{question}
\label{xxque0.9} 
Assume Hypotheses \ref{xxhyp0.1}. Is there a Hopf subalgebra
$H_0\subseteq H$ such that $A^{H_0}=A_G$?
\end{question}

In the classical case, either the Jacobian ${\mathsf j}_{A,H}$
or the reflection arrangement ${\mathsf a}_{A,H}$ completely
determines the collection of reflecting hyperplanes via 
\eqref{E0.3.1} or \eqref{E0.3.2} respectively. In the 
noncommutative case, since $A$ is not a unique factorization 
domain, the decomposition such as \eqref{E0.3.1} (or \eqref{E0.3.2})
is not unique. Consequently, it is not clear how to define 
individual reflecting hyperplanes. We propose the following 
temporary definitions. For any homogeneous element $f\in A$, 
define the set of left (respectively, right) divisors 
of degree 1 of $f$ to be 
\begin{equation}
\label{E0.9.1}\tag{E0.9.1}
{\mathfrak R}^l(f):=\{ \Bbbk v\; \mid \;  v\in A_1, v f_v=f\;\; 
{\text{for some}}\;\; f_v\in A\}.
\end{equation}
and
\begin{equation}
\label{E0.9.2}\tag{E0.9.2}
{\mathfrak R}^r(f):=\{ \Bbbk v\; \mid \;  v\in A_1, f_v  v=f\;\; 
{\text{for some}}\;\; f_v\in A\}.
\end{equation}

Unfortunately, in general (when $H$ is neither commutative nor
cocommutative),
$${\mathfrak R}^l({\mathsf a}_{A,H})\neq 
{\mathfrak R}^r({\mathsf a}_{A,H}),$$
see \eqref{E4.2.6} and \eqref{E4.2.7}.

Some further results related to other invariants (e.g. the homological determinant,
pertinency, and the Nakayama automorphism) are stated as corollaries to
Theorem  \ref{xxthm2.4}. 

This paper is organized as follows. Section \ref{xxsec1} reviews
some basic material. We define and study the Jacobian and the 
reflection arrangement in Section \ref{xxsec2}. In Section \ref{xxsec3} 
we focus on the discriminant. 
In Section \ref{xxsec4}, we give some non-trivial examples with some 
details. 

\subsection*{Acknowledgments}
The authors thank Akira Masuoka and Dan Rogalski for useful conversations 
on the subject, thank Luigi Ferraro, Frank Moore and Robert Won for
sharing their unpublished versions of \cite{FKMW1, FKMW2} and 
thank the referee for his/her very careful reading and valuable comments.
J.J. Zhang was partially supported by the US National Science 
Foundation (No. DMS-1700825).

\section{Preliminaries}
\label{xxsec1}

In this section we recall some basic concepts and fix some notation 
that will be used throughout.

An algebra $A$ is called {\it connected graded} if
$$A=\Bbbk \oplus A_1\oplus A_2\oplus \cdots$$
and $1\in A_0$, $A_iA_j\subseteq A_{i+j}$ for all $i,j\in {\mathbb N}$.
We say $A$ is {\em locally finite} if $\dim_{\Bbbk} A_i<\infty$ for 
all $i$. The {\it Hilbert series} of $A$ is defined to be
$$h_A(t)=\sum_{i\in {\mathbb N}} (\dim_{\Bbbk} A_i)t^i.$$
The {\it Gelfand-Kirillov dimension} (or {\it GKdimension}) 
of a connected $\mathbb{N}$-graded, locally-finite
algebra $A$ is defined to be
$$
\GKdim (A)=\limsup_{n\to\infty} 
\frac{\log (\sum_{i=0}^{n} \dim_{\Bbbk} A_i)}{\log(n)},
$$
see \cite[Chapter 8]{MR}, \cite{KL}, or \cite[p.1594]{StZ}.

The algebras that replace the commutative polynomial rings
are the so-called Artin-Schelter regular algebras \cite{AS}. We recall the 
definition below.

\begin{definition}
\label{xxdef1.1}
A connected graded algebra $A$ is called {\it Artin-Schelter Gorenstein} 
(or {\it AS Gorenstein}, for short) if the following conditions hold:
\begin{enumerate}
\item[(a)]
$A$ has injective dimension $d<\infty$ on
the left and on the right,
\item[(b)]
$\Ext^i_A(_A\Bbbk,_AA)=\Ext^i_{A}(\Bbbk_A,A_A)=0$ for all
$i\neq d$, and
\item[(c)]
$\Ext^d_A(_A\Bbbk,_AA)\cong \Ext^d_{A}(\Bbbk_A,A_A)\cong \Bbbk(\bfl)$ for some
integer $\bfl$. Here $\bfl$ is called the {\it AS index} of $A$.
\end{enumerate}
If in addition,
\begin{enumerate}
\item[(d)]
$A$ has finite global dimension, and
\item[(e)]
$A$ has finite Gelfand-Kirillov dimension,
\end{enumerate}
then $A$ is called {\it Artin-Schelter regular} (or {\it AS
regular}, for short) of dimension $d$.
\end{definition}

Let $M$ be an $A$-bimodule, and let $\mu, \nu$ be algebra 
automorphisms of $A$. Then ${^\mu M^\nu}$ denotes the 
induced $A$-bimodule such that ${^\mu M^\nu}=M$ 
as a $\Bbbk$-space, and where
$$a * m * b=\mu(a)m\nu(b)$$
for all $a,b\in A$ and $m\in {^\mu M^\nu}(=M)$.
Let $1$ denote also the identity map of $A$. We use
${^\mu M}$ (respectively, ${M^\nu}$)
for ${^\mu M^1}$ (respectively, ${^1 M^\nu}$).

Let $A$ be a connected graded finite dimensional algebra. We say 
$A$ is a {\it Frobenius} algebra if there is a nondegenerate 
associative bilinear form 
$$\langle -,- \rangle: A\times A \to \Bbbk,$$
which is graded of degree $-\bfl$, or equivalently, there is 
an isomorphism $A^*\cong A(-\bfl)$ as graded left (or right) 
$A$-modules. There is a (classical) graded Nakayama automorphism 
$\mu\in \Aut(A)$ such that 
$\langle a,b \rangle=\langle \mu(b), a\rangle$
for all $a,b\in A$. Further, $A^*\cong {^\mu A^1}(-\bfl)$
as graded $A$-bimodules. A connected graded AS Gorenstein algebra 
of injective dimension zero is exactly a connected graded 
Frobenius algebra. The Nakayama automorphism can be defined 
for certain classes of infinite dimensional algebras; see the 
next definition.

\begin{definition}
\label{xxdef1.2}
Let $A$ be an algebra over $\Bbbk$, and let $A^e = A \otimes A^{op}$.
\begin{enumerate}
\item[(1)]
$A$ is called {\it skew Calabi-Yau} (or {\it skew CY}, for short) if
\begin{enumerate}
\item[(a)]
$A$ is homologically smooth, that is, $A$ has a projective resolution 
in the category $A^e$-Mod that has finite length and such that each 
term in the projective resolution is finitely generated, and
\item[(b)]
there is an integer $d$ and an algebra automorphism $\mu$ of $A$ 
such that
\begin{equation}
\label{E1.2.1}\tag{E1.2.1}
\Ext^i_{A^e}(A,A^e)=\begin{cases} 0 & i\neq d\\
{^1 A^\mu} & i=d,\end{cases}
\end{equation}
as $A$-bimodules, where $1$ denotes the identity map of $A$.
\end{enumerate}
\item[(2)]
If \eqref{E1.2.1} holds for some algebra automorphism $\mu$ 
of $A$, then $\mu$ is called the {\it Nakayama automorphism} 
of $A$, and is usually denoted by $\mu_A$. 
\item[(3)]
We call $A$ {\it Calabi-Yau} (or {\it CY}, for short) if 
$A$ is skew Calabi-Yau and $\mu_A$ is inner (or equivalently, 
$\mu_A$ can be chosen to be the identity map after changing 
the generator of the bimodule ${^1 A^\mu}$).
\end{enumerate}
\end{definition}

If $A$ is connected graded, the above definition should be made 
in the category of graded modules and \eqref{E1.2.1}
should be replaced by
\begin{equation}
\label{E1.2.2}\tag{E1.2.2}
\Ext^i_{A^e}(A,A^e)=\begin{cases} 0 & i\neq d\\
{^1 A^\mu}(\bfl) & i=d,\end{cases}
\end{equation}
where ${^1 A^\mu}(\bfl)$ is the shift of the graded $A$-bimodule 
${^1 A^\mu}$ by degree $\bfl$.

We will use local cohomology later.
Let $A$ be a locally finite ${\mathbb N}$-graded algebra and 
$\fm$ be the graded ideal $A_{\geq 1}$. Let $A$-$\GrMod$ denote 
the category of ${\mathbb Z}$-graded left $A$-modules. For 
each graded left $A$-module $M$, we define
$$\Gamma_{\fm}(M) =\{ x\in M\mid A_{\geq n} x=0 \; 
{\text{for some $n\geq 1$}}\;\}
=\lim_{n\to \infty} \Hom_A(A/A_{\geq n}, M)$$
and call this the {\it $\fm$-torsion submodule} of $M$. It is 
standard that the functor $\Gamma_{\fm}(-)$ is a left 
exact functor from $A$-$\GrMod$ to itself. Since this category 
has enough injectives, the $i$th right derived functors, denoted by
${\text{H}}^i_{\fm}$ or $R^i\Gamma_{\fm}$, are defined and called 
the {\it local cohomology functors}. Explicitly, one has 
$${\text{H}}^i_{\fm}(M)=R^i\Gamma_{\fm}(M)
:=\lim_{n\to \infty} \Ext^i_A(A/A_{\geq n}, M).$$ 
See \cite{AZ, VdB} for more details.

The Nakayama automorphism of a noetherian AS Gorenstein algebra can be 
recovered by using local cohomology \cite[Lemma 3.5]{RRZ2}:
\begin{equation}
\label{E1.2.3}\tag{E1.2.3}
R^d \Gamma_{\fm} (A)^*\cong {^\mu A^1}(-\bfl)
\end{equation}
where $\bfl$ is the AS index of $A$. 

The following notation will be used throughout.

\begin{notation}[$G, \{p_i\}, \{p_g\}, A_g$] 
\label{xxnot1.3} 
Let $H$ denote a semisimple Hopf algebra. Since $\Bbbk$ is 
algebraically closed, the Artin-Wedderburn Theorem implies that 
$H$ has a decomposition into a direct sum of matrix algebras
\begin{equation}
\label{E1.3.1}\tag{E1.3.1}
H=M_{r_1}(\Bbbk)\oplus M_{r_2}(\Bbbk)\oplus \cdots 
\oplus M_{r_{N-1}}(\Bbbk)\oplus M_{r_N}(\Bbbk)
\end{equation}
with 
\begin{equation}
\label{E1.3.2}\tag{E1.3.2}
1=r_1=\cdots =r_n < r_{n+1} \leq \cdots \leq 
r_i\leq r_{i+1} \leq \cdots \leq r_N.
\end{equation}
Each block $M_{r_i}(\Bbbk)$ corresponds to a simple left $H$-module, denoted by
$S_i$. Then $\{S_i\}_{i=1}^{N}$ is the complete list of simple left $H$-modules 
and $\dim_{\Bbbk} S_i=r_i$ for all $i$. The center of $H$ is a direct sum of 
$N$ copies of $\Bbbk$, each of which corresponds to a block $M_{r_i}(\Bbbk)$. 
Since $H$ is a Hopf algebra, $r_1=1$. Further we can assume that $M_{r_1}=
\Bbbk \inth$ where $\inth$ is the integral of $H$. Each copy of 
$M_{r_i}(\Bbbk)=\Bbbk$, for $i=1,\cdots, n$, gives rise to a central idempotent 
in $H$, which is denoted by $p_i$. Let $I_{com}$ be the ideal of $H$ generated 
by commutators $[a,b]:=ab-ba$ for all $a,b\in H$. Then 
\begin{equation}
\label{E1.3.3}\tag{E1.3.3}
I_{com}=M_{r_{n+1}}(\Bbbk)\oplus \cdots \oplus M_{r_{N-1}}(\Bbbk)
\oplus M_{r_N}(\Bbbk)
\quad {\text{and}}\quad H/I_{com}=\Bbbk^{\oplus n}.
\end{equation}
It is well-known that $I_{com}$ is a Hopf ideal, and consequently, 
$H_{ab}: =H/I_{com}$ is a commutative Hopf algebra. Since $\Bbbk$ 
is algebraically closed, $H_{ab}$ is the dual of a group algebra 
$\Bbbk G$. By \eqref{E1.3.3}, the order of $G$ is $n$. There is 
another way of interpreting $G$. Let $K$ be the dual Hopf algebra 
of $H$, and let $G(K)$ be the group of grouplike elements in $K$.
Then $G$ is naturally isomorphic to $G(K)$, and we can identify $G$ 
with $G(K)$. For every grouplike element $g\in G(K)$, the 
correspondence idempotent in $H_{ab}$ is denoted by $p_{g}$. Then 
the Hopf algebra structure of $H_{ab}$ is given in \cite[p.61]{KKZ4}. 
Let $e$ be the unit or identity element of the group $G$ (later, the 
identity in $G$ is also denoted by $1_{G}$ or $1$). Lifting the idempotent 
$p_{g}\in H_{ab}$ from $H_{ab}$ to the corresponding central idempotent 
in $H$, still denoted by $p_{g}$, we have, in $H$,
\begin{equation}
\label{E1.3.4}\tag{E1.3.4}
p_{g}p_{h}=\begin{cases} p_{g} & g=h,\\ 0& g\neq h, \end{cases}
\quad {\text{and}}\quad \sum_{g\in G} p_{g}\neq 1_{H}, \quad {\text{unless}}\;\; 
n=N,
\end{equation}
and
\begin{equation}
\label{E1.3.5}\tag{E1.3.5}
\Delta(p_{g})=\sum_{h\in G} p_{h}\otimes p_{h^{-1}g}+ X_g, \quad
{\text{and}}\quad
\epsilon(p_{g})=\begin{cases} 1_{\Bbbk} & g=e,\\ 0& g\neq e, \end{cases}
\end{equation}
where $X_g$ is in $I_{com}\otimes H+H\otimes I_{com}$. Since
$I_{com}$ is a Hopf ideal, we also have $\Delta(I_{com})
\subseteq I_{com}\otimes H+H\otimes I_{com}$. Note that $\{p_{g}\}_{g\in G}$
agrees with the idempotents $\{p_i\}_{i=1}^n$. By the duality between 
$H$ and $K$, the idempotent in $H$ corresponding to the integral of $H$ 
is $p_1$ where $1\in K$ is the identity element (or the unit element 
$1_G$ of the group $G$). In other words, $p_1=\inth$. Note that $p_1$ 
is also the first central idempotent corresponding to the 
decomposition \eqref{E1.3.1}.
\end{notation}

Let $A$ be a connected graded algebra and let $H$ be a semisimple 
Hopf algebra acting on $A$ homogeneously and inner-faithfully 
\cite[Definition 1.5]{CKWZ1} such that $A$ is an $H$-module algebra. 
For each idempotent $p_i$, where $i=1,\cdots, n,\cdots, N$, we write 
$A_{p_{i}}=p_i\cdot A$. Then there is a natural decomposition
\begin{equation}
\label{E1.3.6}\tag{E1.3.6}
A=\oplus_{i=1}^N A_{p_i}
\end{equation}
following from the fact $1_H=\sum_{i=1}^N p_i$. Each $p_i$, for 
each $i=1,\cdots,n$, equals $p_g$, for some $g\in G$, and we write 
$$A_{g}:=p_g\cdot A=\{a\in A\mid p_g \cdot a=a\}.$$ 
We recall a definition.

\begin{definition} \cite[Definition 3.2]{KKZ4}
\label{xxdef1.4}
Suppose $H$ acts homogeneously and inner-faithfully on a noetherian 
Artin-Schelter regular domain $A$ that is an $H$-module algebra such 
that the fixed subring $A^H$ \eqref{E0.1.2} is again Artin-Schelter 
regular. Then we say that $H$ acts on $A$ as a {\it reflection Hopf 
algebra} or {\it reflection quantum groups}. By abuse of language, 
sometimes we just say that $H$ is a {\it reflection Hopf algebra} without 
mentioning $A$. If, further, $\hdet^{-1}=\hdet$, then $H$ is called a
{\it true reflection Hopf algebra}.
\end{definition}

\begin{lemma} 
\label{xxlem1.5}
Retain the notation above, and consider $A$ as a $K$-comodule
algebra where $\rho: A \rightarrow A \otimes K$ is the right 
coaction. 
\begin{enumerate}
\item[(1)]
For each $g\in G$, $A_{g}=\{a\in A\mid \rho(a)= a\otimes g\}$.
\item[(2)]
$A_{g} A_{h}\subseteq A_{gh}$ for all $g,h\in G$.
\item[(3)]
Let $A_G$ be $\bigoplus_{g\in G} A_g$. Then $A_G$ is a subalgebra of $A$.
\item[(4)]
If $A$ is a domain, then $G_0:=\{g\in G\mid A_g\neq 0\}$ is a subgroup of $G$. 
\item[(5)]
Suppose $A$ is a domain and $A_g$ {\rm{(}}for some $g\in G${\rm{)}} 
is a nonzero free module over $A^H$ on the left and the right, then 
$A_g$ is a rank one free module over $A^H$ on the left and the right.
\item[(6)]
Assume Hypotheses \ref{xxhyp0.1}. Then each 
nonzero $A_g$ is a rank one free module over $A^H$ on the left and 
the right.
\end{enumerate}
\end{lemma}

\begin{proof} (1)
Let $\{h_1,\cdots, h_{\alpha}\}$ be a $\Bbbk$-linear basis of $H$ and 
$\{h_1^\ast,\cdots,h_{\alpha}^\ast\}$ be the dual basis of $H^\ast=:K$.
Then the element $\sum_{i} h_i \otimes h_i^\ast$ is independent of the 
choice of $\Bbbk$-linear bases $\{h_i\}_{i=1}^{\alpha}$. Since $A$ is a 
left $H$-module, then $A$ is a right $K$-comodule algebra with coaction 
given by
\begin{equation}
\label{E1.5.1}\tag{E1.5.1}
\rho(a)=\sum_{i=1}^{\alpha} (h_i\cdot a)\otimes h_i^\ast
\end{equation}
for all $a\in A$. 

We pick a nice basis consisting of matrix units that correspond 
to the matrix decomposition \eqref{E1.3.1}, making
$\{p_{g}\}_{g\in G}$  a part of the basis for $H$. Since $\Bbbk G(K)$ is 
the dual Hopf algebra of $H_{ab}=H/I_{com}$, then $g(I_{com})=0$ for each 
$g\in G(K)$. For every $h\in G=G(K)$, it is easy to see that 
$g(p_{h})=\delta_{gh}$. This implies that $\{g\}_{g\in G}$ is a 
part of the corresponding dual basis for $K$. Now the assertion 
follows from \eqref{E1.5.1} and a straightforward calculation.

(2) Let $x\in A_g$ and $y\in A_h$, then $I_{com} A_g=I_{com} A_h=0$ 
implies that $I_{com} x=I_{com}y=0$. By \eqref{E1.3.5},
$p_{gh} (xy)=p_{g}(x) p_{h}(y)=xy$. Thus $xy\in A_{gh}$.

(3) This follows from part (2). 

(4) This follows from part (2) and the fact that $A$ is a domain.

(5) Since $A$ is a domain, $xA_{g}\subseteq A^H$ for every nonzero 
$x\in A_{g^{-1}}$. Thus the rank of $A_g$ over $A^H$ is one.

(6) By \cite[Lemma 3.3(2)]{KKZ4} (where the hypothesis that the
${\rm{char}}\; \Bbbk$ is zero is not necessary), $A_g$ is free 
over $A^H$ on both sides. The assertion follows from part (5).
\end{proof}

\begin{notation}[$\{f_g\}, \phi_g$] 
\label{xxnot1.6}
Let $R$ denote the fixed subring $A^H$. Assume that $H$ is a reflection 
Hopf algebra acting on a noetherian Artin-Schelter regular domain $A$. 
By Lemma \ref{xxlem1.5}(6), each nonzero $A_g$ is of the form 
\begin{equation}
\label{E1.6.1}\tag{E1.6.1}
A_g=f_g R= R f_g
\end{equation} 
where $f_g\in A_g$ is a (fixed) nonzero homogeneous element of lowest 
degree. Note that $f_g$ is unique up to a nonzero scalar in $\Bbbk$.
There is a graded automorphism $\phi_g\in \Aut(R)$ such that
\begin{equation}
\label{E1.6.2}\tag{E1.6.2}
f_g x=\phi_g(x) f_g
\end{equation}
for all $x\in R$ \cite[(E3.5.1)]{KKZ4}. For every pair $(g,h)$ of
elements in $G$, define $c_{g,h}\in R$ such that
\begin{equation}
\label{E1.6.3}\tag{E1.6.3}
f_g f_h= c_{g,h} f_{gh}
\end{equation}
\cite[(E3.5.2)]{KKZ4}. Then $c_{g,h}$ is a normal element in $R$ 
and 
\begin{equation}
\label{E1.6.4}\tag{E1.6.4}
c_{g,h}=f_g f_h f_{fg}^{-1}, \quad
{\text{and}}\quad 
\phi^{-1}_{gh} (c_{g,h})=f_{fg}^{-1}f_g f_h.
\end{equation}
\end{notation}


\begin{lemma}
\label{xxlem1.7} 
Let $H$ be a semisimple Hopf algebra acting on an algebra $A$.
\begin{enumerate}
\item[(1)]
If $M$ is a simple left $H$-module and $N$ a 1-dimensional 
left $H$-module, then both $N\otimes M$ and $M\otimes N$ are 
simple left $H$-modules of dimension equal to $\dim_{\Bbbk} M$.
\item[(2)]
If $M\subseteq A$ is a simple left $H$-module and $0\neq b_g\in A_{g}$
where $A_g$ is defined as in Lemma \ref{xxlem1.5}(1), 
then both $M b_g$ and $b_g M$ (if nonzero) are  simple left $H$-modules
of dimension equal to $\dim_{\Bbbk} M$.
\end{enumerate}
\end{lemma}

\begin{proof}
(1) This follows from the fact that $N\otimes -$ and $-\otimes N$ 
are auto-equivalences of the category of left $H$-modules. 

(2) This follows from the fact that the multiplication map
$\mu: A\otimes A\to A$ is a left $H$-module map. Further, as left
$H$-modules, $b_g M\cong \Bbbk b_g\otimes M$ and $Mb_g\cong M\otimes \Bbbk b_g$
when $b_g M$ and $Mb_g$ are nonzero.
\end{proof}

Fixed an integer $d>0$. Let $\{S_{d,i}\}_{i=1}^{w_d}$ be the complete
list of simple left $H$-modules of dimension $d$. For each $g\in G(K)$,
there are permutations in the symmetric group, 
$\sigma_{g,d}, \tau_{g,d}\in {\mathbb S}_{w_d}$, such that
\begin{equation}
\label{E1.7.1}\tag{E1.7.1}
\Bbbk g\otimes S_{d,i}=S_{d,\sigma_{d,g}(i)}, \quad {\text{and}} \quad
\quad S_{d,i}\otimes \Bbbk g\cong S_{d,\tau_{d,g}(i)}.
\end{equation}
Let $\{p_{d,i}\}_{i=1}^{w_d}$ be the complete list of primitive central 
idempotents of $H$ corresponding to the set $\{S_{d,i}\}_{i=1}^{w_d}$, and 
let ${\mathcal A}_{d,i}=p_{d,i} A$. By Lemma \ref{xxlem1.7}(2), we have 
that 
\begin{equation}
\label{E1.7.2}\tag{E1.7.2}
b_g {\mathcal A}_{d,i}\subseteq {\mathcal A}_{d, \sigma_{d,g}(i)}, 
\quad {\text{and}} \quad 
{\mathcal A}_{d,i} b_{g}\subseteq {\mathcal A}_{d, \tau_{d,g}(i)}
\end{equation}
for all $g,i$.

For every $d$, define 
$${\mathcal A}_{d}:=\bigoplus_{i=1}^{w_d} {\mathcal A}_{d,i}.$$

Let $R$ be an Ore domain. If $M$ is a left $R$-module, the rank 
of $M$ over $R$ is defined to be
$$\rk M:= \dim_{Q} Q\otimes_R M$$
where $Q$ is the total quotient division ring of $R$.

\begin{lemma}
\label{xxlem1.8}
Suppose that $A$ is a domain. Let $\rk$ denote the rank over $A^H$.
Suppose that ${\mathcal A}_{d,1}\neq 0$ for some $d$.
\begin{enumerate}
\item[(1)]
$\rk {\mathcal A}_{d}\geq \rk {\mathcal A}_1$.
\item[(2)]
Suppose there are $(d',i')$ such that ${\mathcal A}_{d',i'}\neq 0$
and that $S_{d,1}\otimes S_{d',i'}$ is a direct sum of 
simple $H$-modules of dimensions $d_1,\cdots,d_s$.
Then 
$$\rk {\mathcal A}_{d,1}\leq \sum_{\alpha=1}^s 
\rk {\mathcal A}_{d_{\alpha}}.$$
\item[(3)]
Suppose there are $(d',i')$ such that ${\mathcal A}_{d',i'}\neq 0$
and that $(\bigoplus_{i=1}^{w_d}S_{d,j})\otimes S_{d',i'}$ is a 
direct sum of simple $H$-modules of dimensions $d_1,\cdots,d_s$.
Then 
$$\rk {\mathcal A}_{d}\leq \sum_{\alpha=1}^s 
\rk {\mathcal A}_{d_{\alpha}}.$$
\item[(4)]
Suppose that $S_{d,1}\otimes S_{d,i}$ is a direct sum of 
1-dimensional $H$-simples for some $i$ such that ${\mathcal A}_{d,i}\neq 0$. 
Then $\rk {\mathcal A}_{d,1}\leq \rk {\mathcal A}_1$.
\item[(5)]
If for any ${\mathcal A}_{d,i}\neq 0$, $S_{d,i}\otimes S_{d,1}$ 
is a direct sum of 1-dimensional $H$-simples,
then $\rk {\mathcal A}_{d}\leq \rk {\mathcal A}_1$.
\end{enumerate}
\end{lemma}

\begin{proof}
(1) Let $0\neq x\in {\mathcal A}_{d,1}$ such that $x$ is in a simple left 
$H$-module $M$. By Lemma \ref{xxlem1.7}(2), 
$${\mathcal A}_1 x \subseteq {\mathcal A}_1 M\subseteq {\mathcal A}_{d}.$$ 
Therefore
$$\rk {\mathcal A}_{d}\geq \rk {\mathcal A}_1 x=\rk {\mathcal A}_1.$$

(2) Let $0\neq x\in {\mathcal A}_{d',i'}$. By the ideas in the 
proof of Lemma \ref{xxlem1.7}(2), 
$${\mathcal A}_{d,1} x\subseteq \bigoplus_{\alpha=1}^s {\mathcal A}_{d_\alpha}.$$ 
Therefore
$$\rk {\mathcal A}_{d,1}= \rk {\mathcal A}_{d,1} x\leq \sum_{\alpha=1}^s 
\rk {\mathcal A}_{d_{\alpha}}.$$

(3) The proof is similar to the proof of part (2).

(4,5) These are consequences of parts (2) and (3).
\end{proof}

The above lemma has some consequences. For example, 
if $H$ has only one simple $S$ of dimension $d$ larger than 1
and $S\otimes S$ is a direct sum of 1-dimensional 
$H$-modules, then ${\mathcal A}_1$ and ${\mathcal A}_{d}$
have the same rank. When $|G| = d^2$ this implies that
$\dim_{\Bbbk} H=2 d^2$ \cite{Ar, AC}.

\begin{definition}
\label{xxdef1.9} 
Retain the notation as in Lemma \ref{xxlem1.5} and let
$G=G(K)$.
\begin{enumerate}
\item[(1)]
The subalgebra $A_{G}$ as defined in Lemma \ref{xxlem1.5}(3)
is called the {\it $G$-component} of $A$. 
\item[(2)]
The $\Bbbk$-vector space $A_{G^c}:= \bigoplus_{i= n+1}^N \; p_i \cdot A$ 
where $n$ and $N$ are defined in \eqref{E1.3.2} is called the 
{\it $G$-complement} of $A$. By Lemma \ref{xxlem1.7}(2),
$A_{G^c}$ is an $A_{G}$-bimodule and there is an $A_{G}$-bimodule
decomposition
$$A=A_{G}\oplus A_{G^c}.$$
\end{enumerate}
\end{definition}

An $A$-bimodule $M$ is called {\it $H$-equivariant} in the 
sense of \cite[Definition 2.2]{RRZ2} if 
$$h\cdot (a mb)=\sum (h_1\cdot a) (h_2 \cdot m) (h_3\cdot b)$$
for all $h\in H$, $a,b\in A$ and $m\in M$. The following
lemma is more or less proved in \cite{RRZ2}.

\begin{lemma}
\label{xxlem1.10}
Let $Y$ be an $H$-equivariant graded $A$-bimodule that is free 
of rank one over $A$ on both sides. Then $Y$ is isomorphic to 
${\mathfrak e} \otimes A$ such that 
\begin{enumerate}
\item[(1)]
$\Bbbk {\mathfrak e}$ is a 1-dimensional left $H$-module and there 
is an $g\in G(K)$ such that $h\cdot {\mathfrak e}\otimes 1=
g(h) {\mathfrak e}\otimes 1$,
\item[(2)]
${\mathfrak e}\otimes 1$ is a generator of the free right
$A$-module $Y$, namely, $({\mathfrak e}\otimes 1)a
={\mathfrak e}\otimes a$ for all $a\in A$,
\item[(3)]
there is a graded algebra automorphism $\mu$ of $A$ such that
$$a ({\mathfrak e}\otimes 1)={\mathfrak e} \otimes \mu(a)$$
for all $a\in A$,
\item[(4)]
$\mu(\Xi^r_{g}(h)\cdot a)=\Xi^l_{g}(h)\cdot \mu(a)$, where 
$\Xi^r_{\pi}$ is the right winding automorphism of $H$ associated 
to $g$, defined to be
\begin{equation}
\label{E1.10.1}\tag{E1.10.1}
\Xi^r_{g}: h \mapsto \sum h_1 g(h_2)
\end{equation} 
for all $h\in H$. 
\end{enumerate}
In this case, we write 
\begin{equation}
\label{E1.10.2}\tag{E1.10.2}
Y=(\Bbbk {\mathfrak e})\otimes {^\mu A^1}.
\end{equation}
When $Y$ is $R^d \Gamma_{\fm}(A)^\ast$ for an AS 
Gorenstein ring $A$, $\mu$ is the Nakayama automorphism of
$A$. 
\end{lemma}

The proof of the above lemma is easy and omitted. If we want to
specify the algebra $A$, \eqref{E1.10.2} can be written as
\begin{equation}
\label{E1.10.3}\tag{E1.10.3}
Y_{A}=(\Bbbk {\mathfrak e}_{A})\otimes {^\mu A^1}.
\end{equation}

\begin{definition}
\label{xxdef1.11}
Suppose a Hopf algebra $H$ acts inner-faithfully and homogeneously on 
a connected graded algebra $A$.  Let $R$ be $A^H$. 
\begin{enumerate}
\item[(1)]
The {\it left covariant module} of the $H$-action on $A$ is defined 
to be 
$$A^{l, cov\; H}:=A/A R_{\geq 1}$$
which is a left $A$ and right $R$-bimodule. 
\item[(2)]
The {\it right covariant module} of the $H$-action on $A$ is defined 
to be 
$$A^{r, cov\; H}:=A/ R_{\geq 1}A$$
which is a right $A$ and left $R$-bimodule. 
\item[(3)]
The {\it covariant algebra} of the $H$-action on $A$ is defined to be the 
factor ring
$$A^{cov\; H}: =A/(R_{\geq 1}).$$
\item[(4)]
We say the $H$-action on $A$ is {\it tepid} if 
$A R_{\geq 1}=R_{\geq 1}A$. In this case we say the covariant ring
$A^{cov\; H}$ is {\it tepid}.
\end{enumerate}
\end{definition}

There are reflection Hopf algebras $H$ such that the $H$-action
on $A$ is not tepid and the covariant ring $A^{cov \; H}$ is not
Frobenius, see Example \ref{xxex4.2}.

\section{The Jacobian and the Reflection Arrangement}
\label{xxsec2}

In this section we will introduce two important concepts for Hopf
algebra actions on Artin-Schelter regular algebras: the Jacobian and the
reflection arrangement. We also study the connection between the
Jacobian and the pertinency ideal. 

As in the previous sections, $H$ is a semisimple Hopf algebra. In 
this section we will use the homological determinant
\cite[Definition 3.3]{KKZ3} in a slightly more general situation.
Assume that $A$ is a noetherian connected graded AS Gorenstein algebra
(which is not necessarily regular). Let $\hdet$ denote both the 
homological determinant $\hdet: H\to \Bbbk$ and the corresponding 
grouplike element in $K$ (in \cite{KKZ3} it is called co-determinant).  
As usual, suppose that $H$ acts on $A$ homogeneously and inner-faithfully.  

To motivate our definition, we first briefly recall some facts in the 
commutative situation. Let $A$ be the commutative polynomial ring 
$\Bbbk[V^{\ast}] =\Bbbk[x_1,\cdots,x_n]$ and $G$ be a finite subgroup 
of ${\text{GL}}(V)$ acting on $A$ naturally. Suppose that $G$ is a 
reflection group and $R:=A^G$ is a polynomial ring, written as 
$\Bbbk[f_1,\cdots,f_n]$. Then the {\it Jacobian} $J$ (also called the 
{\it Jacobian determinant}) of the basic invariants $\{f_1,\cdots,f_n\}$ 
is defined to be 
$$J:=\det \; \left(\frac{\partial f_i}{\partial x_j}\right)_{i,j=1}^n,$$
see \cite[Introduction]{HS}.
It is well-known that $\deg J=-n+\sum_{i=1}^n \deg (f_i)$ and that 
$g\cdot J=(\det g)^{-1} J$ for all $g\in G$, see
\cite[p139]{Sta} or \cite[p.229]{OT}. In the commutative case, we 
have $\hdet=\det$. It is also well-known that $A_{\hdet^{-1}}$ 
is free over $R$ on both sides and the lowest degree of nonzero
elements in $A_{\hdet^{-1}}$ is $-n+\sum_{i=1}^n \deg (f_i)$.
Hence $A_{\hdet^{-1}}=J R=RJ$ \cite[p.139]{Sta}.

A result of Steinberg \cite{Ste, HS} says that the Jacobian determinant 
$J$ in the commutative case is a product of linear forms (with 
multiplicities) that correspond to the reflecting hyperplanes 
\eqref{E0.3.1}. The product of the distinct linear forms, denoted by 
${\mathsf a}$, corresponding to the reflecting hyperplanes, namely, 
the reduced defining equation of the Jacobian determinant \eqref{E0.3.2}, 
has the property that $g\cdot {\mathsf a} =\det(g) {\mathsf a}$ for 
all $g\in G$ and the degree of ${\mathsf a}$ is the lowest degree of 
nonzero elements in $A_{\hdet}$. This means that 
$A_{\hdet}={\mathsf a} R=R{\mathsf a}$, see \cite[Theorem 2.3]{Sta} 
and \cite[p.229]{OT}.

The following definition attempts to mimic these classical concepts 
in the noncommutative setting. See Definition \ref{xxdef0.3} under
Hypotheses \ref{xxhyp0.1}. 

\begin{definition}
\label{xxdef2.1} 
Let $A$ be AS Gorenstein, $\hdet\in K$ be the homological 
determinant of the $H$-action on $A$ and $R=A^H$.
\begin{enumerate}
\item[(1)]
If $A_{\hdet^{-1}}$ is free of rank one over $R$ on both sides,
namely, $A_{\hdet^{-1}}=f_{\hdet^{-1}} R=Rf_{\hdet^{-1}}\neq 0$, 
then the {\it Jacobian} of the $H$-action on $A$ is 
defined to be
$${\mathsf j}_{A,H}:=_{\Bbbk^\times} f_{\hdet^{-1}}\in A.$$
\item[(2)]
If $A_{\hdet}$ is free of rank one over $R$ on both sides and
$A_{\hdet}=f_{\hdet} R=Rf_{\hdet}\neq 0$,
the {\it reflection arrangement} of the $H$-action on $A$ is defined to be
$${\mathsf a}_{A,H}:=_{\Bbbk^\times} f_{\hdet}\in A.$$
\end{enumerate}
\end{definition}

In the above definition we do not assume that the fixed subring $A^H$ 
is Artin-Schelter regular. Next we give some easy examples; in (1) and 
(3) $A^H$ is not  AS regular, but the Jacobian and the reflection 
arrangement are still defined. 

\begin{example}
\label{xxex2.2}
\begin{enumerate}
\item[(1)]
If $\hdet$ is trivial, then both ${\mathsf j}_{A,H}$ and 
${\mathsf a}_{A,H}$ are $1\in A$.
\item[(2)] 
\cite[Example 3.7]{KKZ4}
In \cite{KKZ4} we assume that ${\rm{char}}\; \Bbbk=0$, but, in fact,
it suffices to assume that ${\rm{char}}\; \Bbbk \neq 2$ in this example.
Let $G$ be the dihedral group of order $8$. It is generated 
by $r$ of order $2$ and $\rho$ of order $4$ subject to the 
relation $r\rho =\rho^3 r$. Let $A$ be generated by $x, y, z$
subject to the relations
$$\begin{aligned}
zx &= -xz,\\
yx &= zy,\\
yz &= xy.
\end{aligned}
$$
Then $A$ is an AS regular algebra of global dimension $3$.
Let $H=(\Bbbk G)^{\ast}$ and define the $G$-degree of the 
generators of $A$ as
$$\deg_G(x) = r, \quad \deg_G(y) = r\rho, \quad \deg_G(z) = r\rho^2.$$
Then $\Bbbk G$ coacts on $A$. By \cite[Example 3.7]{KKZ4}, 
the Hopf algebra $H$ acts on $A$ as a (true) reflection Hopf 
algebra and the fixed subring $A^H$ is isomorphic to the polynomial
ring $\Bbbk[t_1,t_2,t_3]$, which is AS regular. 
(Note that $t_1=x^2, t_2=y^2, t_3=z^2$.) 
One can check that $\hdet=\hdet^{-1}=r \rho^3$ (so $H$ is a true 
reflection Hopf algebra) and that 
\begin{equation}
\label{E2.2.1}\tag{E2.2.1}
{\mathsf j}_{A,H}= {\mathsf a}_{A,H}=_{\Bbbk^{\times}} zxy
=_{\Bbbk^{\times}} zyz=_{\Bbbk^{\times}} xyx
=_{\Bbbk^{\times}} xzy=_{\Bbbk^{\times}} yzx=_{\Bbbk^{\times}} yxz,
\end{equation}
which is a product of elements of degree 1. By \cite[Theorem 3.5(2)]{KKZ4}, 
the covariant algebra $A^{cov\; H}$ is always tepid in this setting. 

Let us recall the notation introduced in \eqref{E0.9.1}-\eqref{E0.9.2}. For 
any homogeneous element $f\in A$, define the set of left 
(respectively, right) divisors of degree 1 of $f$ to be 
\begin{equation}
\notag 
{\mathfrak R}^l(f):=\{ \Bbbk v\; \mid \;  v\in A_1, v f_v=f\; 
{\text{for some}} f_v\in A\}.
\end{equation}
and
\begin{equation}
\notag 
{\mathfrak R}^r(f):=\{ \Bbbk v\; \mid \;  v\in A_1, f_v  v=f\; 
{\text{for some}} f_v\in A\}.
\end{equation}
It is clear that ${\mathfrak R}^l({\mathsf j}_{A,H})$
contains $\{\Bbbk x, \Bbbk y, \Bbbk z\}$. By using the fact that $y$ is normal,
one can show (with details omitted) that if $yxz=_{\Bbbk^{\times}}f_1 f_2 f_3$ 
for three elements $f_i$ of degree 1, then $f_1 f_2 f_3$ must be, up to scalars 
on $f_i$, one of the expressions given in \eqref{E2.2.1}. Therefore
\begin{equation}
\label{E2.2.2}\tag{E2.2.2}
{\mathfrak R}^l({\mathsf j}_{A,H})={\mathfrak R}^l({\mathsf a}_{A,H})
={\mathfrak R}^r({\mathsf j}_{A,H})={\mathfrak R}^r({\mathsf a}_{A,H})
=\{\Bbbk x, \Bbbk y, \Bbbk z\}.
\end{equation}
One might consider the set $\{\Bbbk x, \Bbbk y, \Bbbk z\}$ as (linear 
forms of) reflecting hyperplanes. 
\item[(3)]
Let $A$ be the down-up algebra 
$${\mathbb D}(0,1):=\frac{\Bbbk \langle u,d\rangle}{(u^2d-d u^2, ud^2-d^2u)}.$$
Then $A$ is noetherian, AS regular of global dimension 3.
Let $H$ be the Hopf algebra $(\Bbbk G)^\ast$ where $G$ is
the dihedral group of order 8 as in part (2). This is the setting in 
\cite[Example 2.1]{CKZ1}. By \cite[Example 2.1]{CKZ1}, we have 
$\hdet=\hdet^{-1}=\rho^2$. The fixed subring $A^H$ is not AS regular 
but is AS Gorenstein. By \cite[Lemma 2.2(3)]{CKZ1},
the Jacobian and the reflection arrangement of the $H$-action 
on $A$ are
$${\mathsf a}_{A,H}={\mathsf j}_{A,H}=_{\Bbbk^{\times}}
u^2\in A.$$
One can show directly that the covariant algebra $A^{cov\; H}$ is tepid. 
\end{enumerate}
\end{example}

\begin{remark}
\label{xxrem2.3}
The definition of the Jacobian in Definition \ref{xxdef2.1}(1) 
agrees with the Jacobian (determinant) when we consider classical 
reflection groups acting on commutative polynomial rings.
\begin{enumerate}
\item[(1)]
In the commutative case, both ${\mathsf j}_{A,H}$ and 
${\mathsf a}_{A,H}$ are products of linear forms 
\eqref{E0.3.1}-\eqref{E0.3.2}. It is natural to ask, if $A$ is 
generated in degree one, under what hypotheses, are both  
${\mathsf j}_{A,H}$ and ${\mathsf a}_{A,H}$ products
of elements of degree 1?
\item[(2)]
In the commutative case one sees from \eqref{E0.3.1}-\eqref{E0.3.2}
that ${\mathsf a}_{A,H}$ divides ${\mathsf j}_{A,H}$.
Is there a generalization of this statement in the noncommutative 
setting? We will discuss this question in Section \ref{xxsec3} 
(see Theorem \ref{xxthm3.8}(2)).
\item[(3)]
More importantly, the definitions of the Jacobian and the 
reflection arrangement suggest that we should search for a 
generalization of hyperplane arrangements in the 
noncommutative setting.
\item[(4)]
In the classical case, ${\mathsf a}_{A,H}$ is reduced, namely, 
every factor is squarefree in ${\mathsf a}_{A,H}$. What is the 
analogue of this statement? See Example \ref{xxex2.2}(2,3). 
\end{enumerate}
\end{remark}

Next we have a result concerning the existence of
${\mathsf j}_{A,H}$. Let $\pi: H\to \Bbbk$ be an 
algebra homomorphism, namely, $\pi\in K$ is a grouplike
element. Recall from \eqref{E1.10.1} that the right winding 
automorphism of $H$ associated to $\pi$ is defined to be
$$\Xi^r_{\pi}: h \mapsto \sum h_1 \pi(h_2)$$ 
for all $h\in H$. The left winding automorphism 
$\Xi^l_{\pi}$ of $H$ associated to $\pi$ is defined
similarly, and it is well-known that both $\Xi^r_{\pi}$
and $\Xi^l_{\pi}$ are algebra automorphisms of $H$.
For any element $x\in A$, let $\eta_x$ denote the ``conjugation'' 
map
$$\eta_x: a \to x^{-1}a x$$
whenever $x^{-1} a x$ is defined. In particular, this 
map could be defined for all $a$ in a subring of $A$. 
In the following result we do not assume that the $H$-action
on $A$ is inner-faithful.

Recall that
\begin{equation}
\label{E2.3.1}\tag{E2.3.1}
\sum \hdet(h_1)\hdet^{-1}(h_2)=\sum \hdet(h_2)\hdet^{-1}(h_1)=\epsilon(h)
\end{equation}
for every $h\in H$. 

\begin{theorem}
\label{xxthm2.4}
Let $A$ be a noetherian AS Gorenstein algebra. Let
$\hdet$ be the homological determinant of the $H$-action
on $A$. 
\begin{enumerate}
\item[(1)]
\cite[Lemma 3.10]{RRZ2}
Let $\mu$ be the Nakayama automorphism of $A$.
Then, for every $a\in A$ and $h\in H$,
\begin{equation}
\label{E2.4.1}\tag{E2.4.1}
\Xi^l_{\hdet}(h) \cdot \mu(a)=\mu( \Xi^r_{\hdet}(h)\cdot a).
\end{equation}
As a consequence, $\mu(A^H)=A^H$.
\item[(2)]
If $A^H$ is AS Gorenstein, then the Jacobian 
${\mathsf j}_{A,H}$ is defined and
\begin{enumerate}
\item[(a)]
${\mathfrak l}_{A^H}={\mathfrak l}+\deg {\mathsf j}_{A,H}$,
where ${\mathfrak l}$ indicates the respective AS indices 
{\rm{[}}Definition \ref{xxdef1.1}(c){\rm{]}},
\item[(b)]
$\mu_{A^H}=\eta_{{\mathsf j}_{A,H}} \circ \mu$.
\end{enumerate}
\item[(3)]
If ${\mathsf j}_{A,H}$ exists, then $A^H$ is AS Gorenstein.
\item[(4)]
$\mu(A_{\hdet^{-1}})=A_{\hdet^{-1}}$. As a consequence,
if ${\mathsf j}_{A,H}$ exists, then $\mu({\mathsf j}_{A,H})
=_{\Bbbk^{\times}} {\mathsf j}_{A,H}$.
\item[(5)]
$\mu(A_{\hdet})=A_{\hdet}$. As a consequence,
if ${\mathsf a}_{A,H}$ exists, then $\mu({\mathsf a}_{A,H})
=_{\Bbbk^{\times}} {\mathsf a}_{A,H}$.
\item[(6)]
Let $A$ be a domain. Suppose there is a short exact sequence 
of Hopf algebras
$$1\to H_0\to H\to \overline{H}\to 1$$
such that $A^{H}$ and $A^{H_0}$ are AS Gorenstein. Then 
$${\mathsf j}_{A,H}=_{\Bbbk^{\times}} {\mathsf j}_{A,H_0}
{\mathsf j}_{A^{H_0}, \overline{H}}
=_{\Bbbk^{\times}} {\mathsf j}_{A^{H_0}, \overline{H}}
{\mathsf j}_{A,H_0}.$$
\end{enumerate} 
\end{theorem}

\begin{proof} (1) Let $R$ denote $A^H$. 
The first claim is a special case of 
\cite[Lemma 3.10]{RRZ2} when the antipode $S$ of $H$ has 
the property that $S^2$ is the identity. 
(Note that since $H$ is semisimple, $S^2$ is the
identity.) For the consequence, we have, for $h\in H$ and $r\in R$, 

$$\begin{aligned}
h\cdot \mu(r)&=\sum \hdet(h_2) h_1\cdot \mu(r) \hdet^{-1}(h_3)\\
&=\sum \Xi^{r}_{\hdet}(h_1)\cdot \mu(r) \hdet^{-1}(h_2)\\
&=\mu(\sum \Xi^{l}_{\hdet}(h_1) \cdot r \hdet^{-1}(h_2))\\
&=\mu(\sum \hdet(h_1) h_2\cdot r \hdet^{-1}(h_3))\\
&=\mu(\sum \hdet(h_1) \epsilon(h_2) r \hdet^{-1}(h_3))\\
&=\epsilon(h) \mu(r).
\end{aligned}
$$
This implies that $\mu(r)\in R$, and completes the proof of part (1).
In the above computation we used \eqref{E2.3.1}.

We will use the notation introduced in \cite{KKZ3}.
Let ${\rm{H}}_{\fm}^i(A)$ be the $i$th local cohomology 
of $A$ with respect to the graded maximal ideal $\fm:=A_{\geq 1}$. 
Let $(-)^{\ast}$ denote the graded $\Bbbk$-linear dual of a 
graded vector space. Let $d$ be the injective dimension of $A$. 
By \cite[p.3648]{KKZ3} or \eqref{E1.2.3},
$$({\rm{H}}_{\fm}^i(A))^*=\begin{cases} 0 & i\neq d\\
{^\mu A^1}(-{\mathfrak l})=:Y & i=d,\end{cases}$$
and
$$({\rm{H}}_{\fm^{R}}^i(R))^*=\begin{cases} 0 & i\neq d\\
Y\cdot \inth =S(\inth) \cdot Y & i=d.\end{cases}$$
As a consequence, the injective dimension of $R$ is also
$d$ if $R$ is AS Gorenstein. Here $\mu$ is the Nakayama 
automorphism of $A$, and  $\mu_R$ is the Nakayama automorphism of $R$. 
Note that $Y$ has an $A$-bimodule structure with compatible 
$H$-action, or in other words, $Y$ is an $H$-equivariant 
$A$-bimodule in the sense of \cite[Definition 2.2]{RRZ2}, see 
\cite[Lemma 3.2(a)]{RRZ2}.

Using the notation in \cite[(3.2.1) and (3.2.2)]{KKZ3} or in Lemma 
\ref{xxlem1.10}, $Y= (\Bbbk {\mathfrak e})\otimes {^\mu A^1}$ as a 
left $H$-module (as well as graded $A$-bimodule) where 
$\deg ({\mathfrak e})={\mathfrak l}$ and the $H$-action 
on ${\mathfrak e}$ is given by
\begin{equation}
\label{E2.4.2}\tag{E2.4.2}
h\cdot {\mathfrak e}= \hdet(h) {\mathfrak e}
\end{equation}
by \cite[Definition 3.3]{KKZ3}. (In \cite{KKZ3},
the authors used the right $H$-action, one can easily transfer
to the left action by composing with the antipode $S$.)
By \cite[Lemma 2.4(1)]{KKZ3}, there is an 
$R$-bimodule decomposition 
\begin{equation}
\label{E2.4.3}\tag{E2.4.3}
A=R\oplus C
\end{equation}
where $R\subseteq A$ is a graded subalgebra. 
Further, as a left $H$-module, $R$ is a direct sum 
of trivial $H$-modules, and, 
$$R=\{a\in A\mid p_1\cdot a =a\};$$
and
$$C=\{a\in A\mid (1-p_1) \cdot a =a\},$$
where $p_1$ is the idempotent in \eqref{E1.3.1} corresponding to the
integral of $H$.
The decomposition \eqref{E2.4.3} gives rise to a
decomposition of $Y$, as $R$-bimodules,
\begin{equation}
\label{E2.4.4}\tag{E2.4.4}
Y=({\rm{H}}_{\fm}^d(A))^*=
({\rm{H}}_{\fm^R}^d(A))^*
=({\rm{H}}_{\fm^R}^d(R))^* \oplus ({\rm{H}}_{\fm^R}^d(C))^*
\end{equation}
where $({\rm{H}}_{\fm^R}^d(R))^*$ is preserved by the left action 
of $p_1$ and $({\rm{H}}_{\fm^R}^d(C))^*$ is preserved by the
left action of $1-p_1$. Using the fact, $Y=(\Bbbk {\mathfrak e})
\otimes {^\mu A^1}$, we can write 
$$({\rm{H}}_{\fm^R}^d(R))^*
=(\Bbbk {\mathfrak e})\otimes V, \quad {\text{and}} \quad  
({\rm{H}}_{\fm^R}^d(C))^*=(\Bbbk {\mathfrak e})\otimes W$$ 
for some graded $R$-bimodules $V,W$ with 
${^\mu A^1}=V\oplus W$. 

(2) Assume that $R$ is AS Gorenstein. Then the $R$-bimodule
$({\rm{H}}_{\fm^R}^d(R))^*$ is isomorphic to ${^{\mu_R} R^1}
(-{\mathfrak l}_R)$. In particular, 
$({\rm{H}}_{\fm^R}^d(R))^*$ is free of rank one on both sides. 
This implies that $V$ is a free $R$-module of rank one on both sides. 

Since $({\rm{H}}_{\fm^R}^d(R))^*$ is preserved by the left action 
of $p_1$ and $({\rm{H}}_{\fm^R}^d(C))^*$ is preserved by the
left action of $1-p_1$, by \eqref{E2.4.2},
$V$ is preserved by the left action 
of $p_{\hdet^{-1}}$ and $W$ is preserved by the
left action of $1-p_{\hdet^{-1}}$. Thus $V={^\mu A^1}_{\hdet^{-1}}
=A_{\hdet^{-1}}$ where the last equation follows from the fact
that the $H$-action on ${^\mu A^1}$ agrees with the $H$-action 
on $A$. Combining these assertions with ones in the last paragraph, 
we obtain that ${\mathsf j}_{A,H}$ exists.

For the two sub-statements, note that the right $R$-module 
$({\rm{H}}_{\fm^R}^d(R))^*$ is free with a generator
${\mathfrak e}\otimes {\mathsf j}_{A,H}$. Using the notation
introduced in \eqref{E1.10.3}, we have
\begin{equation}
\label{E2.4.5}\tag{E2.4.5}
{\mathfrak e}_{R}={\mathfrak e}\otimes {\mathsf j}_{A,H}={\mathfrak e}_A
\otimes {\mathsf j}_{A,H}.
\end{equation}
Then
$${\mathfrak l}_{R}=\deg {\mathfrak e}_{R}=
\deg ({\mathfrak e}_{A}\otimes {\mathsf j}_{A,H})
=\deg {\mathfrak e}_{A}+\deg {\mathsf j}_{A,H}={\mathfrak l}_{A}
+\deg {\mathsf j}_{A,H}.$$
Hence sub-statement (a) follows.
Considering elements inside $Y:={\mathfrak e}\otimes A$, 
for every $r\in R$, using part (1), we have
$$ \begin{aligned}
r ({\mathfrak e}_{R}\otimes 1)&= r ({\mathfrak e}\otimes {\mathsf j}_{A,H} 1)\\
&={\mathfrak e}\otimes \mu(r){\mathsf j}_{A,H}
={\mathfrak e}\otimes {\mathsf j}_{A,H} 
({\mathsf j}_{A,H}^{-1}\mu(r){\mathsf j}_{A,H})\\
&=({\mathfrak e}_R\otimes 1)
({\mathsf j}_{A,H}^{-1}\mu(r){\mathsf j}_{A,H})
\end{aligned}
$$
which implies that $\mu_{R}(r)=\eta_{{\mathsf j}_{A,H}}\circ 
\mu(r)$; hence we have verified sub-statement (b).

(3) The proof of the converse is similar. Since ${\mathsf j}_{A,H}$
is defined, $V:=A_{\hdet^{-1}}$ is a free $R$-module of rank one on 
both sides. Then $({\rm{H}}_{\fm^R}^d(R))^* 
=(\Bbbk {\mathfrak e})\otimes V$ is a free $R$-module of rank one on 
both sides. By \cite[Lemma 1.7(2)]{RRZ2}, 
$({\rm{H}}_{\fm^R}^d(R))^*$ is isomorphic to ${^{\mu_R} R^1}(-{\mathfrak l}_R)$
for some automorphism $\mu_R$ of $R$ and some integer ${\mathfrak l}$. 
By \cite[Lemma 1.6]{KKZ3}, $R$ is AS Gorenstein.

(4) For $r\in A_{\hdet^{-1}}$ and $h\in H$, we have 
$$\begin{aligned}
h\cdot \mu(r)&=\sum \hdet^{-1}(h_1) \hdet(h_2) h_3 \cdot \mu (r)\\
&=\sum \hdet^{-1}(h_1) \Xi^l_{\hdet}(h_2) \cdot \mu(r)\\
&=\sum \hdet^{-1}(h_1) \mu(\Xi^r_{\hdet}(h_2) \cdot r) 
          \qquad\qquad \qquad {\text{by \eqref{E2.4.1}}}\\
&=\sum \hdet^{-1}(h_1) \mu(\sum \hdet(h_3) h_2\cdot r)\\
&=\sum \hdet^{-1}(h_1) \mu(\sum \hdet(h_3) \hdet^{-1}(h_2) r)\\
&=\sum \hdet^{-1}(h_1) \mu(\sum \epsilon(h_2) r)\\
&=\hdet^{-1}(h) \mu(r).
\end{aligned}
$$
Hence the main assertion follows, and the consequence is clear.

(5) For $r\in A_{\hdet}$ and $h\in H$, we have 
$$\begin{aligned}
h\cdot \mu(r)&=\sum \hdet^{-1}(h_1) \hdet(h_2) h_3 \cdot \mu (r)\\
&=\sum \hdet^{-1}(h_1) \Xi^l_{\hdet}(h_2) \cdot \mu(r)\\
&=\sum \hdet^{-1}(h_1) \mu(\Xi^r_{\hdet}(h_2) \cdot r)
        \qquad\qquad \qquad {\text{by \eqref{E2.4.1}}}\\
&=\sum \hdet^{-1}(h_1) \mu(\sum \hdet(h_3) h_2\cdot r)\\
&=\sum \hdet^{-1}(h_1) \mu(\sum \hdet(h_3) \hdet(h_2) r)\\
&=\sum \hdet^{-1}(h_1) \hdet(h_2) \hdet(h_3) \mu(r)\\
&=\sum \epsilon(h_1) \hdet(h_2) \mu(r)= \hdet(h) \mu(r).
\end{aligned}
$$
Hence the main assertion follows, and the consequence is clear.

(6) Let $r\in H_0$ and $h\in H$. Since $H_0$ is normal, 
$\sum S(h_1) r h_2\in H_0$, and for all $x\in A^{H_0}$,
we have 
$$\sum S(h_1) r h_2 (x)=\epsilon(\sum S(h_1) r h_2) (x)
=\epsilon (r) \epsilon(h) x.$$
Then  
$$rh (x)=\sum h_1 S(h_2) r h_3 (x)=\sum h_1 \epsilon(h_2) \epsilon(r)(x)
=\epsilon(r) (h(x))$$
which implies that $A^{H_0}$ is a left $H$-module algebra.
By the definition, $H_0$-action on $A^{H_0}$ is trivial,
so $\overline{H}$ acts on $A^{H_0}$ naturally and 
$$A^H=(A^{H_0})^{H}=(A^{H_0})^{\overline{H}}.$$
By \eqref{E2.4.5},
$$\begin{aligned}
{\mathfrak e}_{A^H}&={\mathfrak e}_{A}\otimes {\mathsf j}_{A,H}, 
                      \qquad {\text{and}}\\
{\mathfrak e}_{A^H}&={\mathfrak e}_{A^{H_0}}
\otimes {\mathsf j}_{A^{H_0},\overline{H}} \\
&= ({\mathfrak e}_{A}\otimes {\mathsf j}_{A, H_0})
\otimes {\mathsf j}_{A^{H_0},\overline{H}} \\
&= ({\mathfrak e}_{A}\otimes {\mathsf j}_{A, H_0})
{\mathsf j}_{A^{H_0},\overline{H}} \\
&= {\mathfrak e}_{A}\otimes ( {\mathsf j}_{A, H_0}
{\mathsf j}_{A^{H_0},\overline{H}} )
\end{aligned}
$$
inside $Y$. Then 
\begin{equation}
\label{E2.4.6}\tag{E2.4.6}
{\mathsf j}_{A,H}={\mathsf j}_{A, H_0}
{\mathsf j}_{A^{H_0},\overline{H}}.
\end{equation} 
For the second equation,
we use part (4). Since both $A^H$ and $A^{H_0}$ are AS Gorenstein,
by part (4), we have
$\mu({\mathsf j}_{A,H})=_{\Bbbk^{\times}} {\mathsf j}_{A,H}$ and
$\mu({\mathsf j}_{A,H_0})=_{\Bbbk^{\times}} {\mathsf j}_{A,H_0}$.
Applying $\mu$ to the equation ${\mathsf j}_{A,H}={\mathsf j}_{A, H_0}
{\mathsf j}_{A^{H_0},\overline{H}}$, and using the hypothesis that
$A$ is a domain, we obtain that 
\begin{equation}
\label{E2.4.7}\tag{E2.4.7}
\mu({\mathsf j}_{A^{H_0},\overline{H}})
=_{\Bbbk^{\times}}{\mathsf j}_{A^{H_0},\overline{H}}.
\end{equation} 
Applying $\mu_{A^{H_0}}$ to ${\mathsf j}_{A^{H_0},\overline{H}}$ and use 
part (4), we have 
\begin{equation}
\label{E2.4.8}\tag{E2.4.8}
\mu_{A^{H_0}}({\mathsf j}_{A^{H_0},\overline{H}})
=_{\Bbbk^{\times}} {\mathsf j}_{A^{H_0},\overline{H}}.
\end{equation}
Combining \eqref{E2.4.7}, \eqref{E2.4.8} with part (2b), 
\begin{equation}
\notag 
{\mathsf j}_{A^{H_0},\overline{H}}
=_{\Bbbk^{\times}} \eta_{{\mathsf j}_{A,H}}({\mathsf j}_{A^{H_0},\overline{H}})
\end{equation}
or equivalently,
\begin{equation}
\label{E2.4.9}\tag{E2.4.9}
{\mathsf j}_{A^{H_0},\overline{H}}{\mathsf j}_{A,H}=_{\Bbbk^{\times}}
{\mathsf j}_{A,H}{\mathsf j}_{A^{H_0},\overline{H}}.
\end{equation}
Since $A$ is a domain, the combination of \eqref{E2.4.6} and \eqref{E2.4.9}
implies that 
\begin{equation}
\notag 
{\mathsf j}_{A^{H_0},\overline{H}}{\mathsf j}_{A,H_0}=_{\Bbbk^{\times}}
{\mathsf j}_{A,H_0}{\mathsf j}_{A^{H_0},\overline{H}}
\end{equation}
as desired.
\end{proof}

Theorem \ref{xxthm2.4}(6) is useful for the case when $H$ is obtained
by an abelian extension of Hopf algebras. We  wonder
if there is a version of Theorem \ref{xxthm2.4}(6) for ${\mathsf a}_{A,H}$.
Theorem \ref{xxthm2.4}(2b) is a generalization of 
\cite[Theorem 0.6(1)]{KKZ4}.  Though the Jacobian exists,
it is not clear if the reflection arrangement exists 
when $R$ is AS Gorenstein. We have three corollaries, including
the existence of the reflection arrangement when $R$ is AS regular.
The first of the corollaries is Theorem \ref{xxthm0.2}(1,2).

\begin{corollary}
\label{xxcor2.5}
Assume Hypotheses \ref{xxhyp0.1}. Let $R=A^H$ and 
$\xi(t)=h_A(t) (h_R(t))^{-1}$.
\begin{enumerate}
\item[(1)]
Both ${\mathsf j}_{A,H}$ and ${\mathsf a}_{A,H}$ exist.
\item[(2)]
$\deg \xi(t)= \deg {\mathsf j}_{A,H}$.
\item[(3)]
$h_{A^{l,cov\; H}}(t)=h_{A^{r,cov\; H}}(t)=\xi(t)$.
As a consequence,
$$\dim A^{l,cov\; H}=\dim A^{r,cov\; H}=\xi(1),$$ 
where $h_{A^{l,cov\; H}}(t)$ and $h_{A^{r,cov\; H}}(t)$ 
are the Hilbert series of $A^{l,cov\; H}$ and 
$A^{r,cov\; H}$.
\end{enumerate}
\end{corollary}

\begin{proof} (1) By Theorem \ref{xxthm2.4}, the Jacobian 
${\mathsf j}_{A,H}$ exists. In particular, $A_{\hdet^{-1}}\neq 0$.
Since $K$ is finite dimensional, $\hdet$ is a power of 
$\hdet^{-1}$. So $A_{\hdet}\neq 0$. By Lemma \ref{xxlem1.5}(6), $A_{\hdet}$ 
is a free $R$-module of rank one on both sides, and hence by definition,
${\mathsf a}_{A,H}$ exists.

(2) Let $p_A(t)=(h_A(t))^{-1}$ and $p_{R}(t)=(h_{R}(t))^{-1}$. 
By \cite[Proposition 3.1]{StZ}, $\deg p_A(t)={\mathfrak l}_A$
and $\deg p_R(t)={\mathfrak l}_R$. By Theorem \ref{xxthm2.4}(2a),
$$\deg {\mathsf j}_{A,H}={\mathfrak l}_R-{\mathfrak l}_A=\deg p_R(t)
-\deg p_A(t)=\deg \xi(t).$$

(3) Since $A_R$ is a finitely generated free $R$-module, 
$h_A(t)=h_{A^{l,cov\; H}}(t) h_R(t)$, and the assertion follows.
The consequence is clear.
\end{proof}

The next corollary is a rigidity result.

\begin{corollary}
\label{xxcor2.6}
Let $A$ be a noetherian AS Gorenstein algebra with finite GKdimension.
Suppose $H$ acts on $A$ such that $A^H$ is AS Gorenstein.
\begin{enumerate}
\item[(1)]
Suppose $A$ is Cohen-Macaulay. If $\hdet$ is not trivial,
then $\p(A,H)\leq 1$, where $\p(A,H)$ is defined in Definition \ref{xxdef2.8}(3).
\item[(2)]
Suppose that there is no graded ideal $I\subseteq A$ such that
$\GKdim A/I=\GKdim A-1$. Then $\hdet$ is trivial. 
\item[(3)]
If $A$ is projectively simple in the sense of \cite[Definition 1.1]{RRZ1}
and if $\GKdim A\geq 2$, then there is no graded ideal $I\subseteq A$ such that
$\GKdim A/I=\GKdim A-1$.
\end{enumerate}
\end{corollary}

\begin{proof} (1) If $\hdet$ is not trivial, then there is an $R$-bimodule
$C$ such that 
$$A=A^H\oplus A_{\hdet} \oplus C=R\oplus {\mathsf j}_{A,H} R\oplus C$$
[Theorem \ref{xxthm2.4}(2)]. Then $\End_R (A)$ is not ${\mathbb N}$-graded.
Therefore the natural map $A \# H\to \End_R(A)$ cannot be an isomorphism
of a graded algebras. By \cite[Theorem 3.5]{BHZ1}, $\p(A,H)\leq 1$. 

(2) Suppose to the contrary that $\hdet$ is not
trivial. By Theorem \ref{xxthm2.4}, $f:={\mathsf j}_{A,H}
\in A_{\geq 1}$ exists. By definition, $fR=Rf$ inside $A$. Consider 
the $(A,R)$-bimodule $M:=A/Af$ which is finitely generated 
on both sides, we have $\GKdim (M)=\GKdim A-1$. Let $I=\ann_A(_AM)$.
Since $M_R$ is finitely generated, $M=\sum_{i=1}^s m_i R$. Then 
$I=\bigcap \ann_A( m_i)$. For each $i$, $\GKdim (A/\ann_A(m_i))
\leq \GKdim M$. Then $\GKdim A/I\leq \GKdim M$. Since
$I\subseteq Af$, $\GKdim A/I\geq \GKdim M$. Therefore 
$\GKdim A/I=\GKdim M=\GKdim A-1$, a contradiction.

(3) This is clear from the definition of a projectively simple ring
(also called a just-infinite ring).
\end{proof}

The third corollary puts some constraints on the homological determinant 
$\hdet$. Recall from \cite[p. 318]{RRZ2} that an AS Gorenstein algebra 
$A$ is called {\it $r$-Nakayama}, for some $r\in \Bbbk^{\times}$, the 
Nakayama automorphism of $A$ is of the form
\begin{equation}
\label{E2.6.1}\tag{E2.6.1}
\mu: a\to r^{\deg a} a
\end{equation}
for all homogeneous element $a\in A$. For example, every Calabi-Yau 
AS regular algebra is $1$-Nakayama.

\begin{corollary}
\label{xxcor2.7}
Let $A$ be a noetherian AS Gorenstein algebra that is $r$-Nakayama for some
$r\in \Bbbk^{\times}$. {\rm{(}}We need  only that $\Xi^r_{\hdet}(h)$ 
is a stable map of $\mu$-isotropy classes.{\rm{)}}
\begin{enumerate}
\item[(1)]
Assume that the $H$-action on $A$ is faithful. 
Then $\hdet$ is a central element in $G(K)$.
As a consequence, if the center of $G(K)$ is trivial,
then $\hdet$ is trivial and $A^H$ is AS Gorenstein.
\item[(2)]
Suppose that $A$ is an AS regular domain and that $G\neq \{1\}$ is a 
finite group with trivial center {\rm{(}}e.g. $G$ is non-abelian 
simple{\rm{)}}. 
If $H:=(\Bbbk G)^{\ast}$ acts on $A$ inner-faithfully and 
homogeneously such that $A$ is an $H$-module algebra, then 
$\hdet$ is trivial and $H$ is not a reflection Hopf algebra 
in the sense of Definition \ref{xxdef1.4}.
\end{enumerate}
\end{corollary}

\begin{proof}
(1) Under hypothesis of $\mu$ being $r$-Nakayama and the fact that 
$\mu$ is a graded algebra homomorphism, \eqref{E2.4.1} becomes
$$\Xi^{r}_{\hdet}(h) \cdot a= \Xi^{l}_{\hdet}(h)\cdot a$$
for all $a\in A$ and $h\in H$. Since the $H$-action is faithful,
we have $\Xi^{r}_{\hdet}(h)=\Xi^{l}_{\hdet}(h)$ for all
$h\in H$. Applying $g\in G(K)$ to the above equation, 
we obtain that $(g\circ \hdet)(h)=(\hdet\circ g)(h)$. Thus
$\hdet$ commutes with all elements $g\in G(K)$. This shows the
main assertion, and the consequence is clear.

(2) By Lemma \ref{xxlem1.5}(4), $G_0:=\{g\in g\mid A_{g}\neq 0\}$
is a subgroup of $G$.  Since the $H$-action on $A$ is inner-faithful,
the $K$-coaction on $A$ is inner-faithful. Thus, $G_0=G$. This 
implies that $H$-action on $A$ is in fact faithful. By part (1),
$\hdet$ is trivial. By \cite[Theorem 0.6]{CKWZ1}, $A^H$ is not
AS regular, hence $H$ is not a reflection Hopf algebra.
\end{proof}

\begin{definition}
\label{xxdef2.8}
Let $H$ act on $A$ and $\inth$ be the integral of $H$. 
\begin{enumerate}
\item[(1)] 
The {\it pertinency ideal} of the $H$-action on $A$ is defined to be
$${\mathcal P}_{A,H}:=(A\# H)(1\# \inth) (A\# H)\subseteq A\# H.$$
\item[(2)] \cite[Definition 1.4]{HZ}
The {\it radical ideal} of the $H$-action on $A$ is defined to be
$${\mathfrak r}_{A,H}:={\mathcal P}_{A,H}\cap A$$
identifying $A$ with $A\# 1\subseteq A\# H$.
\item[(3)] \cite[Definition 0.1]{BHZ1} 
The {\it pertinency} of the $H$-action on $A$ is defined to be
$$\p(A,H) := \GKdim(A\#H) - \GKdim(A\#H/\mathcal P_{A,H}).$$
\end{enumerate}
\end{definition}

The {\it radical ideal} of a group $G$-action on an algebra $A$ 
was introduced in \cite[Definition 1.4]{HZ} using pertinence
sequences. By the proof of \cite[Proposition 2.4]{HZ},
that definition agrees with Definition \ref{xxdef2.8}(2) when $H$ is a 
group algebra.

Under some mild hypotheses, we will show that the radical 
ideal is essentially the Jacobian of the $H$-action on $A$
when $H$ is a reflection Hopf algebra. For simplicity,
let $m$ stand for $\hdet^{-1}$ following the notation of \cite{KKZ4}. 

From now on until Theorem \ref{xxthm2.12}, let $H=(\Bbbk G)^\ast$ 
for some finite group $G$. Assume that ${\rm{char}}\; \Bbbk=0$.
Then the integral $\inth$ of $H$ is $p_{1}$ where $1$ is the identity 
of $G$. Since $H=\bigoplus_{g\in G} \Bbbk p_{g}$, we have 
$A=\oplus_{g\in G} A_g$ where $A_g=p_{g}\cdot A$. By using the 
comultiplication given in \eqref{E1.3.5}, one easily checks that 
the following equations hold.

\begin{lemma}
\label{xxlem2.9} Let $H=(\Bbbk G)^\ast$, $g,h\in G$ and $b_h\in A_h$. Then 
\begin{enumerate}
\item[(1)]
$(b_{h}\# 1)(1\# p_{g})=b_h\# p_{g}$.
\item[(2)]
$(1\# p_g)(b_h\# 1)=b_h\# p_{h^{-1}g}$.
\item[(3)]
$(1\# p_{hg})(b_h\# 1)=b_h\# p_{g}$.
\end{enumerate}
\end{lemma}

\begin{lemma}
\label{xxlem2.10} Let $\inth$ be the integral of $H=(\Bbbk G)^\ast$. Then 
$$(A\# 1)\cap (A\# H)(1\#\inth)(A\# H)=(\bigcap_{g\in G} 
A A_{g})\# 1.$$
As a consequence,
$${\mathfrak r}_{A,H}=\bigcap_{g\in G} A A_{g}.$$
\end{lemma}

\begin{proof}
We compute
$$\begin{aligned}
(A\# H)(1\#\inth)(A\# H)&=(\sum_{h} A\# p_h)
(1\# p_{1_G})(\sum_{i,j} A_{i}\# p_j)\\
&=(A\# 1)(1\# p_{1_G})(\sum_{i,j} A_{i}\# p_j)\\
&=(A\# 1)(\sum_{i,j} A_{i}\# p_{i^{-1}}p_j)\\
&=\sum_i AA_i\# p_{i^{-1}}.
\end{aligned}
$$
If $x\in (A\# 1)\cap (A\# H)(1\#\inth)(A\# H)$, 
then $x=y\# 1=y\# \sum_{i} p_{i^{-1}}$ for $y\in A$.
By the above computation, $y\in AA_{i}$ for all 
$i\in G$. Thus $y\in \bigcap_{g\in G} A A_{g}$
as required. 
\end{proof}

\begin{lemma}
\label{xxlem2.11} 
Assume Hypotheses \ref{xxhyp0.4}. Let $m:=\hdet^{-1}\in G$.
\begin{enumerate}
\item[(1)]
For each $g\in G$, there is a nonzero $f_g\in A$ such $A_g=R f_g=f_g R$.
\item[(2)]
For each $g\in G$, there is an $h\in G$ such that $f_h f_g=_{\Bbbk^{\times}}
f_m$.
\item[(3)]
$$\bigcap_{g\in G} A A_{g}=\bigcap_{g\in G} A f_{g}=A f_{m}=f_{m} A.$$
\end{enumerate}
\end{lemma}

\begin{proof} 
(1) By \cite[Theorem 3.5(1)]{KKZ4}, for each $g\in G$,
$A_g=R f_g=f_g R$ for some homogeneous element $0\neq f_g\in A$.

(2) By \cite[Theorem 3.5(2)]{KKZ4}, the covariant
algebra $A^{cov H}$ [Definition \ref{xxdef1.11}] 
is the quotient algebra $A/I$ where 
$I=\oplus_{g\in G} (A^H)_{\geq 1} f_{g}$, and $A^{cov H}$
is Frobenius. Further $A^{cov H}$ has a $\Bbbk$-basis
$\{\overline{f_{g}}\}_{g\in G}$. Since $A^{cov H}$ is graded
and Frobenius, for every $g$, there is
an $h\in G$ such that $\overline{f_h} \; \overline{f_g}=a \overline{f_m}$
for some $0\neq a\in \Bbbk$. Then $hg=m$ and $f_h f_g=a f_{m}$. 

(3) As a consequence of part (2), $Af_{m}\subseteq A f_g$ for all $g$. 
Therefore $\bigcap_{g\in G} A f_{g}=A f_{m}$. By 
\cite[Theorem 0.5(1)]{KKZ4}, $f_m$ is a normal element.
Then $Af_{m}=f_{m}A$. This finishes the proof.
\end{proof}

Now we prove Theorem \ref{xxthm0.5}, which is 
Theorem \ref{xxthm2.12}(2) below. Following \cite{KKZ4}, 
let 
\begin{equation}
\label{E2.11.1}\tag{E2.11.1}
{\mathfrak G}:=\{h\in G \mid \deg f_h=1\}.
\end{equation}
(In \cite{KKZ4}, this set is denoted by 
${\mathfrak R}$.)

\begin{theorem}
\label{xxthm2.12}
Assume Hypotheses \ref{xxhyp0.4}. 
\begin{enumerate}
\item[(1)]
The radical ideal ${\mathfrak r}_{A,H}$ is a principal ideal of $A$
generated by ${\mathsf j}_{A,H}$.
\item[(2)]
Both ${\mathsf j}_{A,H}$ and ${\mathsf a}_{A,H}$ are products of
elements in degree 1 of the form $f_h$.
\item[(3)]
${\mathsf a}_{A,H}$ divides ${\mathsf j}_{A,H}$ from the left and the right.
\end{enumerate}
\end{theorem}

\begin{proof} (1) The assertion follows from  Lemmas 
\ref{xxlem2.10} and \ref{xxlem2.11}(1).

(2) By \cite[Theorem 3.5(5)]{KKZ4}, the covariant algebra 
$A^{cov H}$ is generated by elements $\{f_h \mid h\in {\mathfrak G}\}$.
Using the $G$-grading and the fact that $A^{cov H}$ is a skew 
Hasse algebra \cite[Definition 2.3(2)]{KKZ4},  every $f_g$ is 
a product of $f_{h_1}\cdots f_{h_s}$ if $g=h_1 \cdots h_s$ 
where $s=l_{\mathfrak G}(g)$ \cite[Definition 2.1]{KKZ4}. 
In particular, both ${\mathsf j}_{A,H}$ and ${\mathsf a}_{A,H}$ 
are products of elements elements in $\{f_h\mid h\in {\mathfrak G}\}$.

(3) See proof of Lemma \ref{xxlem2.11}(2).
\end{proof}

Note that, in general, Theorem \ref{xxthm2.12}(1) fails, see 
\eqref{E4.2.12}-\eqref{E4.2.13}.  Motivated by the above 
result, we have the following remarks and questions,
which can be viewed as a continuation of Remark \ref{xxrem2.3}.

\begin{remark}
\label{xxrem2.13}
Assume Hypotheses \ref{xxhyp0.1}. 
\begin{enumerate}
\item[(1)]
What is the connection between ${\mathfrak r}_{A,H}$ and 
${\mathsf j}_{A,H}$? The relation between them is not obvious, but we believe
that ${\mathfrak r}_{A,H}$ is contained in $A{\mathsf j}_{A,H}$.
See Lemma \ref{yylem3.13} for a partial result. 
\item[(2)]
As in Remark \ref{xxrem2.3}(2), we ask:
does ${\mathsf a}_{A,H}$ divide ${\mathsf j}_{A,H}$ 
(from the left and the right)? The answer is yes, see 
Theorem \ref{xxthm3.8}(2). As a consequence,
${\mathfrak R}^l({\mathsf a}_{A,H})$ is a subset of
${\mathfrak R}^l({\mathsf j}_{A,H})$. This suggests another 
question: does the equation
$${\mathfrak R}^l({\mathsf a}_{A,H})={\mathfrak R}^l({\mathsf j}_{A,H})$$
always hold?

On the other hand, we will give an example where
${\mathfrak R}^l({\mathsf j}_{A,H})
\neq {\mathfrak R}^r({\mathsf j}_{A,H})$, see \eqref{E4.2.6}-\eqref{E4.2.7} in
Example \ref{xxex4.2}.

One question related to this inequality is: do we have an isomorphism 
$\phi$ such that $\phi({\mathfrak R}^l({\mathsf j}_{A,H}))
={\mathfrak R}^r({\mathsf j}_{A,H})$ (respectively, 
$\phi({\mathfrak R}^l({\mathsf a}_{A,H}))
={\mathfrak R}^r({\mathsf a}_{A,H})$)? 
\item[(3)]
In the classical case,
$\deg {\mathsf a}_{A,H}$ is the number of reflecting 
hyperplanes and $\deg {\mathsf j}_{A,H}$ is the number
of pseudo-reflections. What are the meanings of
$\deg {\mathsf j}_{A,H}$ and $\deg {\mathsf a}_{A,H}$
in the noncommutative case?
\item[(4)]
Suppose that $A$ is generated in degree one.
Are ${\mathsf a}_{A,H}$ and ${\mathsf j}_{A,H}$ products 
of elements of degree 1? If yes, are these products of 
elements in ${\mathfrak R}^l({\mathsf a}_{A,H}) \cup
{\mathfrak R}^l({\mathsf a}_{A,H})$?
\end{enumerate}
Further assume that $H$ is $(\Bbbk G)^{\ast}$ and that $A$ is 
generated in degree 1.
\begin{enumerate}
\item[(5)]
It follows from \cite[Theorem 0.4]{KKZ4} that ${\mathfrak G}$ 
can be considered as a subset of both 
${\mathfrak R}^l({\mathsf j}_{A,H})$ and
${\mathfrak R}^r({\mathsf j}_{A,H})$. As a consequence,
$|{\mathfrak G}|\leq |{\mathfrak R}^l({\mathsf j}_{A,H})|$.
\item[(6)]
Is the $\deg {\mathsf a}_{A,H}=|{\mathfrak G}|$? 
For example, in Example \ref{xxex2.2}(2) 
${\mathfrak G} =  \{r, r \rho, r \rho^2\}$ and 
$\deg {\mathsf a}_{A,H}=3.$
\item[(7)]
Does the set $\{f_h \mid h\in {\mathfrak G}\}$ 
coincide with ${\mathfrak R}^l({\mathsf a}_{A,H})$?
In the ideal situation, we should call 
${\mathfrak R}^l({\mathsf a}_{A,H})$ 
the collection of ``reflecting hyperplanes''. 
In Example \ref{xxex2.2}(2) both the ``reflecting 
hyperplanes'' and the set 
$\{f_h \mid h\in {\mathfrak G}\}$
are basically $\{\Bbbk x,\Bbbk y,\Bbbk z\}$.
See Lemma \ref{xxlem4.1}(2) for a case when $H$ is not
$(\Bbbk G)^{\ast}$. 
\end{enumerate}
\end{remark}

The Jacobian is defined even when $H$ is not a reflection 
Hopf algebra and so in Example \ref{xxex2.2}(1,3) we note 
the following.

\begin{example}
\label{xxex2.14} 
\begin{enumerate}
\item[(1)]
If the $H$-action on $A$ has trivial homological determinant,
then ${\mathsf j}_{A,H}=1$, but the radical ideal 
${\mathfrak r}_{A,H}$ is not the whole algebra $A$. As a 
consequence ${\mathfrak r}_{A,H} \subsetneq ({\mathsf j}_{A,H})$.
\item[(2)]
In Example \ref{xxex2.2}(3) it follows from \cite[Lemma 2.2]{CKZ1} 
that 
$${\mathfrak r}_{A,H} =u^2(d R+udu R)A\cap (u^3R+dudR)A
\subsetneq u^2 A=({\mathsf j}_{A,H}).$$ 
\end{enumerate}
\end{example}

\section{Discriminants}
\label{xxsec3}

Geometrically the discriminant locus of a reflection group $G$ acting on $\Bbbk[V]$
is the image of reflecting hyperplanes in the corresponding 
affine quotient space \cite[Proposition 6.106]{OT}. 
Algebraically, the discriminant of $G$ is the product of 
Jacobian and reflection arrangement (as an element in the 
fixed subring $\Bbbk[V]^G$). In the noncommutative case,
we can define the discriminant as the product of the Jacobian 
and the reflection arrangement. However, the product of two 
elements in a noncommutative ring is dependent on the order 
of these elements. Therefore we make the following definitions.

\begin{definition}
\label{xxdef3.1} 
Suppose that both the Jacobian ${\mathsf j}_{A,H}$ and the reflection 
arrangement ${\mathsf a}_{A,H}$ exist, namely, 
$A_{\hdet^{-1}}={\mathsf j}_{A,H} R=R{\mathsf j}_{A,H}$ and 
that $A_{\hdet}={\mathsf a}_{A,H} R=R{\mathsf a}_{A,H}$ where $R=A^H$.
\begin{enumerate}
\item[(1)]
The {\it left discriminant} of the $H$-action on $A$, or the
{\it left $H$-discriminant} of $A$, is defined to be
$$\delta^l_{A,H}:=_{\Bbbk^\times} {\mathsf a}_{A,H} \; {\mathsf j}_{A,H}
\in R.$$
\item[(2)]
The {\it right discriminant} of the $H$-action on $A$, or the
{\it right  $H$-discriminant} of $A$, is defined to be
$$\delta^r_{A,H}:=_{\Bbbk^\times} {\mathsf j}_{A,H} \; {\mathsf a}_{A,H}
\in R.$$
\item[(3)]
If $\delta^l_{A,H}=_{\Bbbk^{\times}} \delta^r_{A,H}$, then 
$\delta^r_{A,H}$ is called {\it discriminant} of the $H$-action 
on $A$, or the {\it $H$-discriminant} of $A$, and denoted by
$\delta_{A,H}$.
\item[(4)]
The ideal ${\mathfrak r}_{A,H}\cap R$ of $R$ is called the
{\it $H$-dis-radical}, and denoted by $\Delta_{A,H}$.
\end{enumerate}
\end{definition}

We consider the following list of hypotheses that are weaker than 
Hypotheses \ref{xxhyp0.1}.

\begin{hypothesis}
\label{xxhyp3.2} 
Assume the following hypotheses:
\begin{enumerate}
\item[(a)] 
$A$ is a noetherian connected graded AS Gorenstein algebra.
\item[(b)] 
Hypotheses \ref{xxhyp0.1}(b,c).
\item[(c)]
$A$ is a free module over $R$ on both sides. 
\item[(d)]
$G_0:=\{ g\in G\mid A_g\neq 0\}$ is a subgroup of $G(K)$
and each $A_g$, for $g\in G_0$, is a free $R$-module of rank one on
both sides.
\end{enumerate}
\end{hypothesis}

Continuing Example \ref{xxex2.2}, up to scalars, in Definition \ref{xxdef3.1} (1) 
$\delta_{A,H} = 1$, in (2) $\delta_{A,H}=z^2x^2y^2$, 
and in (3) $\delta_{A,H} = u^4$.  Note that in part (3) 
$\delta_{A,H}$ exists although Hypothesis \ref{xxhyp3.2}(c) 
above is not satisfied. It is possible that Hypothesis 
\ref{xxhyp3.2}(c) can be weakened in part (2) of the following
lemma.

\begin{lemma}
\label{xxlem3.3}
\begin{enumerate}
\item[(1)]
Assume Hypotheses \ref{xxhyp0.1}. Then Hypotheses \ref{xxhyp3.2}
holds.
\item[(2)]
Assume  Hypotheses \ref{xxhyp3.2}. Then $R$ is AS Gorenstein and both 
${\mathsf j}_{A,H}$ and ${\mathsf a}_{A,H}$ exist.
\end{enumerate} 
\end{lemma}

\begin{proof} (1) Nothing needs to be proved for Hypotheses 
\ref{xxhyp3.2}(a,b). Part (c) is \cite[Lemma 3.3.(2)]{KKZ4}. 
Part (d) is Lemma \ref{xxlem1.5}(4,6).
(2) By \cite[Theorem 11.65]{Ro}, there is a standard spectral 
sequence for change of rings 
$$\Ext^p_A(\Tor_q^R(A,M),A)\Rightarrow \Ext^{p+q}_R(M,A)$$
for all left $R$-modules $M$. Since $A$ is finitely generated 
and free over $R$ on both sides, the above spectral sequence
collapses to
$$\Ext^p_A(A\otimes_R M, A)=\Ext^p_R(M, A).$$
This implies that $R$ has finite injective dimension and 
$\Ext^d_R(\Bbbk, R)$ is finite dimensional. By \cite[Theorem 0.3]{Zh},
$R$ is AS Gorenstein. By Theorem \ref{xxthm2.4}, ${\mathsf j}_{A,H}$ is
defined, or equivalently, $\hdet^{-1}\in G_0$. 
Since $G_0$ is a group and $\hdet\in G_0$, $A_{\hdet}$ is 
free of rank one on both sides by Hypothesis \ref{xxhyp3.2}(d). Then 
${\mathsf a}_{A,H}$ is defined. 
\end{proof}

The following lemma shows the existence of the discriminant 
under Hypotheses \ref{xxhyp3.2}.

\begin{lemma}
\label{xxlem3.4}
Assume Hypotheses \ref{xxhyp3.2}.  Let $g\in G_0$; then $A_g$ 
and $A_{g^{-1}}$ are free of rank one over $R$ on both sides. Let 
$f_g$ and $f_{g^{-1}}$ be the generators of $A_{g}$ and $A_{g^{-1}}$, 
respectively, over $R$; then the following properties hold.  
\begin{enumerate}
\item[(1)]
Every $f_{g^{-1}}f_{g}$ is a normal element in $R$. 
In particular, both $\delta^l_{A,H}$ and $\delta^r_{A,H}$ 
are normal elements in $R$.
\item[(2)]
If $g'\in G_0$, then 
$A_{g'}\cap A_{G} f_{g}= Rf_{g'}\cap R f_{g'g^{-1}} f_{g}$ and
$A_{g'}\cap f_g A_{G}= f_{g'}R \cap f_{g} f_{g^{-1}g'}R$.
As a consequence, if $f_g$ is a normal element in $A_{G}$ and 
$f_{g'}$ divides $f_g$ from the left and the right, then 
$f_g f_{g^{-1}g'}=_{\Bbbk^{\times}} f_{g'g^{-1}} f_g$.
In particular, if $f_g$ is a normal element in $A_{G}$, then 
$f_g f_{g^{-1}}=_{\Bbbk^{\times}} f_{g^{-1}} f_g$.
\item[(3)]
$R \delta^l_{A,H}=R \cap A_{G} {\mathsf j}_{A,H}$ and 
$\delta^r_{A,H} R=R \cap {\mathsf j}_{A,H}A_{G}$.
\item[(4)]
If ${\mathsf j}_{A,H}$ is a normal element in $A_{G}$, then 
$\delta_{A,H}$ is well-defined.
\end{enumerate}
\end{lemma}

\begin{proof} Since $g\in G_0$, $g\in G(K)$ such that $A_g\neq 0$. 
Since $G_0$ is a group [Hypothesis \ref{xxhyp3.2}(d)], 
$A_{g^{-1}}$ is nonzero. By Hypothesis \ref{xxhyp3.2}(d), $A_g$ 
and $A_{g^{-1}}$ are free of rank one over $R$ on both sides. 

(1) Clearly $f_{g^{-1}}f_{g}$ is an element in $R$ 
for every $g\in G_0$. It follows from \eqref{E1.6.2} that
$$f_{g^{-1}}f_{g} x= f_{g^{-1}} \phi_g(x) f_g=
(\phi_{g^{-1}}\circ \phi_g)(x)f_{g^{-1}}f_{g}.$$
Hence $f_{g^{-1}}f_{g}$ is a normal element in $R$. 

(2) We will use Lemma \ref{xxlem1.7}. For $g,g'\in G_0$, we compute
$$\begin{aligned}
A_{g'}\cap A_{G} f_{g} &= Rf_{g'}\cap (\sum_{d=1,i} {\mathcal A}_{d,i} f_{g})\\
&=Rf_{g'}\cap (\sum_{h\in G} R f_{h}f_{g})\\
&=Rf_{g'}\cap (R f_{g'g^{-1}}f_{g}\oplus \sum_{h\neq g'g^{-1}} R f_{h} f_{g})\\
&=Rf_{g'}\cap R f_{g'g^{-1}}f_{g}.
\end{aligned}
$$
Similarly, we have $A_{g'}\cap f_{g} A_{G}= 
f_{g'}R \cap f_{g} A_{G}=f_{g'}R \cap f_{g}f_{g^{-1}g'} R$.
If $f_{g}$ is a normal element in $A_{G}$, then $A_{G}f_{g}=f_{g}A_{G}$. 
Since $f_{g'}R=R f_{g'}$ and since $f_{g'}$ divides $f_g$ from 
the left and the right, we have
$$\begin{aligned}
f_{g}f_{g^{-1}g'} R&=f_{g'}R \cap f_{g}f_{g^{-1}g'} R\\
&=f_{g'}R \cap f_{g} A_{G}\\
&=A_{g'}\cap A_{G} f_{g}\\
&=A_{g'}\cap f_{g} A_{G}\\
&=Rf_{g'}\cap R f_{g'g^{-1}}f_{g}\\
&=R f_{g'g^{-1}}f_{g}.
\end{aligned}
$$
Then $f_{g}f_{g^{-1}g'}=_{\Bbbk^{\times}} f_{g'g^{-1}}f_{g}$.
Let $g'=1$, we obtain that
$f_g f_{g^{-1}}=_{\Bbbk^{\times}} f_{g^{-1}} f_g$.

(3) The assertion follows from part (2) by taking $g'=1$ 
and $g=\hdet^{-1}$.

(4) The assertion follows from parts (2,3) and the fact that 
$f_1=1$ divides $f_{\hdet^{-1}}$ trivially.
\end{proof}

The following is Theorem \ref{xxthm0.2}(3) in a special case.

\begin{theorem}
\label{xxthm3.5}
Assume Hypotheses \ref{xxhyp0.1}.
Suppose that ${\rm{char}}\; \Bbbk=0$ and 
that $H$ is commutative, namely, $H=(\Bbbk G)^{\ast}$. 
\begin{enumerate}
\item[(1)]
The discriminant $\delta_{A,H}$ is defined, namely, 
$$\delta^l_{A,H}=\delta^r_{A,H}=\delta_{A,H}.$$
\item[(2)]
$\Delta_{A,H}$ is the principal ideal of
$R$ generated by $\delta_{A,H}$.
\end{enumerate}
\end{theorem}

\begin{proof} 
(1) By \cite[Theorem 0.5(1)]{KKZ4}, ${\mathsf j}_{A,H}=
f_{\hdet^{-1}}$ is a normal element in $A$. Now the assertion
follows from Theorem \ref{xxthm2.12}(1) and Lemma \ref{xxlem3.4}(4). 

(2) The assertion
follows from Theorem \ref{xxthm2.12}(1) and Lemma \ref{xxlem3.4}(3). 
\end{proof}

\begin{remark}
\label{xxrem3.6}
Here we make some remarks and ask some questions before we prove 
one of the main results in this section, namely, Theorem
\ref{xxthm3.8}.
\begin{enumerate}
\item[(1)]
Assuming Hypotheses \ref{xxhyp0.1} or \ref{xxhyp3.2}, 
is $\delta_{A,H}$ always defined? The answer is YES, see 
Theorem \ref{xxthm3.8}. We might further ask:
is $\Delta_{A,H}=(\delta_{A,H})$? This is not true, see 
Example \ref{xxex4.2}. 
\item[(2)]
Note that in the commutative case, $R/(\delta_{A,H})$ is 
always reduced. So we ask the following questions in the 
noncommutative case: assuming Hypotheses \ref{xxhyp0.1}. 
is the factor ring $R/(\delta_{A,H})$ semiprime?

In Example \ref{xxex2.2}(2), 
$\delta_{A,H}=_{\Bbbk^{\times}}x^2 y^2 z^2=t_1 t_2 t_3$
and $R/(\delta_{A,H})$ is semiprime and reduced.
\item[(3)]
In the commutative case, ${\mathsf a}_{A,H}$ is reduced 
in $A$. When ${\mathsf a}_{A,H}$ is normal in $A$ 
(which is not always true by Example \ref{xxex4.2}), 
we can ask if $A/({\mathsf a}_{A,H})$ is semiprime.

In Example \ref{xxex2.2}(2), 
${\mathsf a}_{A,H}=_{\Bbbk^{\times}}xyx$
and $A/({\mathsf a}_{A,H})$ is semiprime, but contains 
nonzero nilpotent elements.
\item[(4)]
Suppose that $A$ is generated in degree 1. We ask if
$${\mathfrak R}^l({\mathsf a}_{A,H})={\mathfrak R}^l({\mathsf j}_{A,H})
={\mathfrak R}^l(\delta_{A,H})=\{ \Bbbk f_g\mid g\in {\mathfrak G}\}?$$
A similar question can be asked for ${\mathfrak R}^r$.
\end{enumerate}
\end{remark}

To prove the existence of $\delta_{A,H}$, we need
to recall some terminology introduced in Section \ref{xxsec1}.
For every left $A$-module, 
$${\text{H}}^i_{\fm}(M)=\lim_{n\to \infty} \Ext^i_A(A/A_{\geq n}, M).$$
The local cohomology functors are defined similar for right
$A$-modules $M$. When $M$ is an $A$-bimodule that is finitely 
generated on both sides, 
then ${\text{H}}^i_{\fm}(M)$ can be computed as a left $A$-module
or a right $A$-module (the result is the same). If $R$ is a subring
of $A$ such that $A$ is finitely generated over $R$ on both sides, 
then ${\text{H}}^i_{\fm}(M)$ can be computed by considering $M$ as 
a module over $R$. In the next lemma, we might calculate $\HH(M)$
in the category of graded right $R$-modules. 

\begin{lemma}
\label{xxlem3.7}
Assume Hypotheses \ref{xxhyp3.2}. Let $d$ be the injective
dimension of $A$. Suppose $g\in G_0$. 
\begin{enumerate}
\item[(1)]
Then the left action 
$p_g: A\to A$ is a right $R$-module map such that 
it decomposes into
\begin{equation}
\label{E3.7.1}\tag{E3.7.1}
p_g: A\xrightarrow{\tilde{p}_g} f_g R \xrightarrow{p^{-1}_g} A.
\end{equation}
\item[(2)]
Applying $\HH(-)$ to \eqref{E3.7.1}, $\HH(p_g)$ is the left action 
of $p_{g^{-1}}$ on the module $\HH(A)$, which decomposes into
$$\HH(A)\xrightarrow{\HH(p^{-1}_g)} 
\HH(f_g R)\xrightarrow{\HH(\tilde{p}_g)}  \HH(A).$$
\item[(3)]
Let $l_{f_g}$ be the left multiplication of element 
$f_g$ on $A$, then $\HH(l_{f_g})$ is the right multiplication
by $f_g$ on $\HH(A)$.
\item[(4)]
The composition
\begin{equation}
\label{E3.7.2}\tag{E3.7.2}
\Phi:=p_g \circ _{f_g} \circ p_1: A\xrightarrow{p_1} A\xrightarrow{l_{f_g}}
A\xrightarrow{p_g} A
\end{equation}
maps $R$, as a component of $A$ \eqref{E1.3.6}, to $f_g R=Rf_g$ 
and other component of $A$ to zero. The restriction
of the map $\Phi$ on $R$ with image $f_g R$ is an isomorphism
of right $R$-modules.
\item[(5)]
After applying $\HH(-)$ to \eqref{E3.7.2}, 
$$\HH(\Phi)=\HH(p_g \circ _{f_g} \circ p_1): \HH(A)\xrightarrow{p_{g^{-1}}} 
\HH(A)\xrightarrow{r_{f_g}}
\HH(A)\xrightarrow{p_1} \HH(A)$$
maps $\HH(f_g R)$ to $\HH(R)$ and other component of $\HH(A)$ to zero
where $r_{f_g}$ is the right multiplication by $f_g$.
\item[(6)]
$\HH(f_g R)=\{ x\in \HH(A)\mid h\cdot x= g(h) x\}
= {\mathfrak e}\otimes f_{\hdet^{-1} g^{-1}}R$ and 
$\HH(R) =\{ x\in \HH(A)\mid h\cdot x= \epsilon(h) x\}
={\mathfrak e}\otimes f_{\hdet^{-1}} R$.
\item[(7)]
$f_{\hdet^{-1}}=_{\Bbbk^{\times}}f_{(\hdet^{-1}g^{-1})}f_g$.
\item[(8)]
$f_{\hdet^{-1}}=_{\Bbbk^{\times}}f_h f_{(h^{-1}\hdet^{-1})}$
and $\deg f_{\hdet^{-1}}\geq \deg f_h$ for all $h\in G_0$. 
\item[(9)]
$f_{\hdet^{-1}}$ is a normal element in $A_G$.
\item[(10)]
$A_{G}/(R_{\geq 1})$ is Frobenius. 
\end{enumerate}
\end{lemma}

\begin{proof}
(1) In this case $A_g=f_g R$ which is free of rank one over $R$
on both sides. Since the left action of $p_g$ is a right $R$-module map,
we obtain a right $R$-module decomposition of the map $p_g$.

(2) Note that $\HH(A)$ is an $H$-equivariant $A$-bimodule 
where the left $H$-action comes from the natural right 
$H$-action  on $\HH(A)$ \cite[Lemma 3.2(a)]{RRZ2}.
By definition \cite[(E2.4.1)]{RRZ2} for $i=0$, $\HH(p_g)$ 
is $p_{g^{-1}}$. The decomposition follows from \eqref{E3.7.1}.

(3) Again this follows from \cite[Lemma 3.2(a)]{RRZ2} and 
its proof. 

(4) This follows from the decomposition of $A$ and 
Lemma \ref{xxlem1.7}(2).

(5) Note that $\HH(A)$ is an $H$-equivariant $A$-bimodule.
By the proof of \cite[Lemma 3.2(a)]{RRZ2},  $\HH(p_1)$, 
$\HH(l_{f_g})$ and $\HH(p_{g})$ are $p_1$, $r_{f_{g}}$
and $p_{g^{-1}}$ (by part (2)). The assertion follows.

(6) Note that $g, \hdet\in G_0$ which is a finite group.
By part (5), $\HH(f_g R)$ is the image of the idempotent $p_{g^{-1}}$.
Hence 
$$\begin{aligned}
\HH(f_g R)&=\{ x\in \HH(A)\mid p_{g^{-1}} \cdot x= x\}\\
&=\{ x\in \HH(A)\mid h \cdot x= g^{-1}(g) x; \; \forall \; h\in H\}\\
&= (\Bbbk {\mathfrak e})\otimes R_{\hdet^{-1} g^{-1}}\\
&= (\Bbbk {\mathfrak e})\otimes f_{\hdet^{-1} g^{-1}}R.
\end{aligned}
$$
This proves the first equation. The second equation is a 
consequence by taking $g=1$.

(7) By part (5),  $\HH(\Phi)$, considered as a map from
$\HH(f_g R)$ to $\HH(R)$, is the right
multiplication by $f_g$. By part (6), this map agrees with
$$r_{f_g}:  (\Bbbk {\mathfrak e})\otimes f_{(\hdet^{-1} g^{-1})}R
\to (\Bbbk {\mathfrak e})\otimes f_{\hdet^{-1}}R.$$
Since $\HH(\Phi)$ is an isomorphism, 
we obtain that $f_{(\hdet^{-1} g^{-1})}R f_{g}=f_{\hdet^{-1}}R$.
Hence the assertion follows. 

(8) The first assertion follows by taking $g=h^{-1}\hdet^{-1}$.
The second assertion is clear.

(9) Every element in $A_G$ is a linear combination of $f_g r$
for some $g\in G_0$ and $r\in R$. Then, by part (8),
$$\begin{aligned}
f_{\hdet^{-1}} (f_g r)&=_{\Bbbk^{\times}} f_{(\hdet^{-1} g\hdet)} 
f_{(\hdet^{-1} g^{-1})} f_g r\\
&=_{\Bbbk^{\times}} f_{(\hdet^{-1} g\hdet)} f_{\hdet^{-1}} r\\
&=_{\Bbbk^{\times}} (f_{(\hdet^{-1} g\hdet)} 
\phi_{\hdet^{-1}}(r) )f_{\hdet^{-1}}.
\end{aligned}
$$
The assertion follows.

(10) Let $F:=A_{G}/(R_{\geq 1})$. Then $F=\bigoplus_{g\in G_0} \Bbbk f_g$
with multiplication satisfying part (7) or (8). For every element $x\in F$,
write $x=\sum c_g f_g$ with $c_g\neq 0$. Pick $g$ so that $\deg f_g$ is smallest
among all $g$ such that $c_g\neq 0$. Then 
$$f_{(\hdet^{-1}g^{-1})} x= c_g f_{(\hdet^{-1}g^{-1})} f_g
=_{\Bbbk^{\times}} f_{\hdet^{-1}}$$
which implies that $F$ is Frobenius.
\end{proof}

Now we are ready to prove Theorems \ref{xxthm0.2}(3) and 
\ref{xxthm0.6}. Following \eqref{E2.11.1}, we define
\begin{equation}
\label{E3.7.3} \tag{E3.7.3}
{\mathfrak R}({\mathsf j}_{A,H}):=\{\Bbbk f_g \mid g\in G_0, \deg f_g=1\}.
\end{equation}

\begin{theorem}
\label{xxthm3.8} 
Assume Hypotheses \ref{xxhyp3.2}. 
\begin{enumerate}
\item[(1)]
${\mathsf j}_{A,H}\; {\mathsf a}_{A,H}=_{\Bbbk^{\times}} {\mathsf a}_{A,H} 
\;{\mathsf j}_{A,H}$. As a consequence, the 
discriminant $\delta_{A,H}$ of the $H$-action is defined.
\item[(2)]
${\mathsf a}_{A,H}$ divides ${\mathsf j}_{A,H}$.
\item[(3)]
${\mathfrak R}({\mathsf j}_{A,H})$ is a subset of both 
${\mathfrak R}^l({\mathsf j}_{A,H})$ and ${\mathfrak R}^r({\mathsf j}_{A,H})$.
\item[(4)]
Assuming the hypotheses of Theorem \ref{xxthm0.5}, then
${\mathfrak R}({\mathsf j}_{A,H})=\{\Bbbk f_g \mid f_g\in {\mathfrak G}\}$
where ${\mathfrak G}$ is defined in \eqref{E2.11.1}.
\end{enumerate}
\end{theorem}

\begin{proof} (1)
By Lemma \ref{xxlem3.7}(9), ${\mathsf j}_{A,H}=f_{\hdet^{-1}}$
is a normal element in $A_G$. The assertions follow from 
Lemma \ref{xxlem3.4}(2,4) by setting $g=\hdet^{-1}$.

(2) This is Lemma \ref{xxlem3.7}(7,8). 

(3) This follows from Lemma \ref{xxlem3.7}(7,8).

(4) This is clear.
\end{proof}

Now we are to prove Theorems \ref{xxthm0.2}, \ref{xxthm0.6} 
and \ref{xxthm0.8}. 

\begin{proof}[Proof of Theorem \ref{xxthm0.2}]
(1,2) This is Corollary \ref{xxcor2.5}(1).

(3) This is Theorem \ref{xxthm3.8}(1).
\end{proof}

Theorem \ref{xxthm0.6} is a consequence of Lemma \ref{xxlem3.3}(1)
and Theorem \ref{xxthm3.8}(2). The next theorem is 
Theorem \ref{xxthm0.8}.

\begin{theorem}
\label{xxthm3.9} 
Assume Hypotheses \ref{xxhyp0.1}. Suppose $R$ is Auslander 
regular. Then $A_G$ is AS Gorenstein and ${\mathsf j}_{A,H}
=_{\Bbbk^{\times}} {\mathsf j}_{A_G, (\Bbbk G)^{\ast}}$ in $A_G$.
\end{theorem}

\begin{proof} 
By Hypotheses \ref{xxhyp0.1}, $A$ is a domain, and hence so is $A_G$.

Since $R$ is AS regular, it is trivially AS Cohen-Macaulay in 
the sense of \cite[Definition 0.1]{JZ}. Since $A_G$ is a finitely
generated free module over $R$, it  also is AS Cohen-Macaulay.
Therefore the hypotheses of \cite[Theorem 6.1(1$^\circ$)]{JZ} hold, and the
hypotheses of \cite[Theorem 6.1(3$^\circ$)]{JZ} hold because
$R$ is AS regular, see \cite[Proposition 5.5]{JZ}. By the proof 
of \cite[Proposition 5.7]{JZ}, using the fact that $R$ is 
Auslander regular, we see that the hypotheses of \cite[Theorem 6.1(2$^\circ$)]{JZ} hold.
Combining the facts that $R$ is AS regular and $A_G/(R_{\geq 1})$
is Frobenius [Lemma \ref{xxlem3.7}(10)], we obtain that
the Hilbert series of $A_G$ satisfies
$$h_{A_G}(t)=\pm t^m h_{A_G}(t^{-1}).$$
Now the first assertion follows from \cite[Theorem 6.1]{JZ}.

For the second assertion, note that $B:=A_G$ satisfies Hypotheses
\ref{xxhyp3.2}. It is clear that $B^{(\Bbbk G)^{\ast}}=A^H=R$. 
Let $f'_{g}$ be the generator of $B_{g}$ as defined  in \eqref{E1.6.1}.
Then $f'_g=f_g$ for all $g\in G$. By Lemma \ref{xxlem3.7}(7), both
$f'_{\hdet^{-1}}$ and $f_{\hdet^{-1}}$ (with different meanings 
of $\hdet^{-1}$) have the highest degree among
$\{f'_g\mid g\in G\}$ and $\{f_g\mid g\in G\}$. Thus 
$f'_{\hdet^{-1}}=f_{\hdet^{-1}}$. This is equivalent to 
${\mathsf j}_{A_G, (\Bbbk G)^{\ast}}={\mathsf j}_{A,H}$ by 
definition.
\end{proof}

Next we prove Theorem \ref{xxthm0.7}. The
discriminant has been an important tool in number theory 
and algebraic geometry for many years. The discriminant of 
a reflection group is a fundamental invariant of 
 reflection group actions. Next we will compare the 
$H$-discriminant in the noncommutative case 
[Definition \ref{xxdef3.1}(3)] to the noncommutative 
discriminant over a central subalgebra, which was used in 
recent studies of automorphism groups and locally nilpotent 
derivations \cite{BZ, CPWZ1, CPWZ2}. 

If $I$ is an ideal of a commutative ring, let $\sqrt{I}$
denote the prime radical ideal of $I$.

\begin{theorem}
\label{xxthm3.10}
Assume Hypotheses \ref{xxhyp0.1}. Further assume that  
\begin{enumerate}
\item[(a)]
${\rm{char}}\; \Bbbk=0$,
\item[(b)]
$H=(\Bbbk G)^\ast$, and
\item[(c)]
$R:=A^H$ is central in $A$. 
\end{enumerate}
Let $dis(A/R)$ be the discriminant defined in \cite[Definition 1.3(3)]{CPWZ1}.
Then 
\begin{enumerate}
\item[(1)]
$dis(A/R)=_{\Bbbk^{\times}} \prod_{g\in G} (f_{g^{-1}} f_{g})$.
\item[(2)]
$$\sqrt{(dis(A/R))}=\sqrt{(\Delta_{A,H})}=\sqrt{(\delta_{A,H})}$$
as ideals of $R$.
\end{enumerate}
\end{theorem}

\begin{proof} (1) Since $A=\oplus_{g\in G} f_{g} R$,
$A$ can be embedded into the matrix algebra $M_r(R)$
by the left multiplication, where $r=|G|$. For each $g$,
the left multiplication  by $f_{g}$ is 
$$l_{f_g}: f_{h}\mapsto f_g f_h =c_{g,h} f_{gh} \text{ for } 
c_{g,h} \in R \hspace{.4in}  \mbox{(see (E1.6.2))}.$$
If $g\neq e$, $gh\neq g$, then the regular trace of $f_g$
\cite[Example 1.2(3)]{CPWZ1}, denoted by $tr(f_{g})$, is zero. As a 
consequence, we have
\begin{equation}
\label{E3.10.1}\tag{E3.10.1}
tr(f_g f_h)=tr(c_{g,h} f_{gh})=\begin{cases} 0& gh\neq e,\\
c_{h^{-1},h}=f_{h^{-1}}f_{h} & g=h^{-1}.\end{cases}
\end{equation}
By \cite[Definition 1.3(3)]{CPWZ1}, the discriminant
$dis(A/R)$ is the determinant of the matrix 
$$\left( tr(f_{g}f_{h}) \right)_{G\times G}.$$
Using \eqref{E3.10.1}, every row (and every column) contains only 
one nonzero entry, namely, $f_{h^{-1}}f_{h}$. Hence, we have
$$dis(A/R)=_{\Bbbk^{\times}}\prod_{h\in G} f_{h^{-1}}f_{h}.$$
(As an example, note that in Example \ref{xxex2.2}(2) 
$$\prod_{h\in G} f_{h^{-1}}f_{h}=(x^2)(y^2)(z^2)(xyzy)(xzxz)(zyxy)(xzyxzy) 
=_{\Bbbk^{\times}}z^8x^8y^8.)$$

(2) By Theorem \ref{xxthm3.5}, $\Delta_{A,H}$ is the principal 
ideal of $R$ generated by $\delta_{A,H}$. Hence $\Delta_{A,H}=
(\delta_{A,H})$, and it remains to show that 
$\sqrt{(dis(A/R))}=\sqrt{(\delta_{A,H})}$. 

Since $\delta_{A,H}= f_{m^{-1}} f_{m}$, by part (1), 
$\delta_{A,H}$ divides $dis(A/R)$. 
By the proof of Lemma \ref{xxlem2.11}, every $f_{g}$ divides
$f_{m}$ from the left and the right. Hence there are $a,b\in A$
such that $af_{g^{-1}}f_g b=f_{m}^2$. Since $f_{g^{-1}}f_{g}$ 
is in the central subring $R$, we have $f_{g^{-1}}f_{g}ab=f_{m}^2$.
This implies that $f_{g^{-1}}f_{g}$ divides $f_{m}^r\in R$. 
(Note that $f_{m}^2\not\in R$ in general.)
As a consequence, $dis(A/R)$ divides $(f_{m}^{r})^{r}$. 
Finally $(f_{m}^{r})^{r}$ divides $\delta_{A,H}^{r^2}$. 
Therefore
$$\sqrt{(\delta_{A,H})}=\sqrt{dis(A/R)}
=\sqrt{f_{m}^r}$$
as desired.
\end{proof}

Theorem \ref{xxthm0.7} is  Theorem \ref{xxthm3.10}(2).
Note that there are many examples where $R$ is not central in $A$,
even when $R$ is a commutative polynomial ring [Example \ref{xxex4.2}].
Without the hypothesis of $H=(\Bbbk G)^{\ast}$, it is 
easy to construct examples where
$$\sqrt{(\Delta_{A,H})}\neq \sqrt{(\delta_{A,H})},$$
see \eqref{E4.2.14}.

\begin{remark}
\label{xxrem3.11} 
Suppose $H$ is a semisimple Hopf algebra.
\begin{enumerate}
\item[(1)]
Let $G(K)$ be the group of all grouplike elements in $K$.
In general, $\Bbbk G(K)$ is NOT a normal Hopf subalgebra
\cite{Ma1}. 
\item[(2)]
One could ask if $\Bbbk G(K)$ is a normal Hopf subalgebra
under Hypothesis \ref{xxhyp0.1}. This is related to Question
\ref{xxque0.9}.
\item[(3)]
If $\Bbbk G(K)$ is a normal Hopf subalgebra,
then there is a short exact sequence of Hopf algebras
\begin{equation}
\label{E3.11.1}\tag{E3.11.1}
1\to \Bbbk G(K)\to K \to K_0\to 1
\end{equation}
where $K_0=K/(\Bbbk G(K))_{+}$.
There is a dual short exact sequence
\begin{equation}
\label{E3.11.2}\tag{E3.11.2}
1\to H_0 \to H \to (\Bbbk G(K))^\ast\to 1
\end{equation}
where $H_0=(K_0)^{\ast}$. If we further 
assume Hypotheses \ref{xxhyp0.1}, then $A_G=A^{H_0}$ and
Question \ref{xxque0.9} has a positive answer under these 
extra hypotheses. 
\end{enumerate}
\end{remark}

We end this section by providing some results that can be used to compute
the radical ideal of the $H$-action, particularly when $A$ has dimension 2.

\begin{definition}
\label{yydef3.12} Let $H$ be a semisimple Hopf algebra acting on $A$.
\begin{enumerate}
\item[(1)]
If $$\Delta(\inth)=\Delta(p_1)=\sum_{h\in G(K)}
p_{h}\otimes p_{h^{-1}} +X_1$$
where $X_1\in I_{com}\otimes I_{com}$, see \eqref{E1.3.5}, then 
$H$ is called {\it rife}. 
\item[(2)]
Assume Hypotheses \ref{xxhyp0.1}. We say the $H$-action is 
{\it rife} if 
\begin{enumerate}
\item[(a)]
$H$ is rife.
\item[(b)]
${\mathsf j}_{A,H}$ is normal in $A$.
\item[(c)]
the radical ideal  of the $H$-action ${\mathfrak r}_{A,H}$ is generated 
by ${\mathsf j}_{A,H}$.
\end{enumerate}
\end{enumerate}
\end{definition}

By Theorem \ref{xxthm2.12}(1), when $H$ is $(\Bbbk G)^{\ast}$, then 
the $H$-action is rife. Otherwise, the $H$-action may not be rife,
even when $H$ is rife [Example \ref{xxex4.2}].

\begin{lemma}
\label{yylem3.13}
Assume Hypotheses \ref{xxhyp0.1}. If $H$ is rife, 
then ${\mathfrak r}_{A,H}$ is a subspace of $A {\mathsf j}_{A,H}$.
\end{lemma}

\begin{proof} For each $g\in G(K)$, since $H$ is rife, we have
$$\begin{aligned}
(A\# \smallint) (A\# p_{g}) &= 
(A\# 1) (\sum_{h\in G(K)} p_{h^{-1}} A \# p_{h} p_{g} 
+ X_1\cdot (A \# p_{g}))\\
&= (A\# 1) (A_{g^{-1}}\# p_{g})\\
&= AA_{g^{-1}} \# p_{g}\\
&= A f_{g^{-1}} \# p_g
\end{aligned}
$$
If $x\in {\mathfrak r}_{A,H}$, then $x\# 1\in (A\# H)(A\#\inth)(A\# H)$.
Multiplying $1\# p_g$ from the right, $x\# p_g\in (A\# H)(A\#\inth)(A\# p_g)$.
By computation,
$$
(A\# H)(A\#\inth)(A\# p_g)=(A\#\inth)(A\# p_g)=Af_{g^{-1}} \# p_{g},
$$
which implies that $x\in Af_{g^{-1}}$ for all $g\in G(K)$. 
Hence $x\in \bigcap_{g\in G(K)} A f_{g}$, which is a
subspace of $Af_{\hdet^{-1}}=A{\mathsf j}_{A,H}$. 
\end{proof}

For every (left) ideal $I$ in a noetherian algebra $A$, 
let $\overline{I}$ denote the largest ideal containing 
$I$ such that $\overline{I}/I$ is finite dimensional.
The following lemma is well-known.

\begin{lemma}
\label{yylem3.14}
Let $A$ be AS regular of global dimension two.
{\rm{(}}So $A$ is noetherian.{\rm{)}}
Let $I$ be a nonzero graded two-sided ideal. If 
$\overline{I}=I$, then $I$ is a principal
ideal generated by a normal element.  In particular, 
$\overline{I}$ is always a principal ideal generated 
by a normal element. 
\end{lemma}

\begin{proof} Since $\overline{I}=I$, 
$A/I$ is $\fm$-torsionfree, so
${\text{H}}^0_{\fm}(M)=0$, see definition in Section
\ref{xxsec1}. By Auslander-Buchsbaum 
formula \cite[Theorem 3.2]{Jo}, the left $A$-module 
$A/I$ has projective dimension at most 1. Since
$A/I$ is not projective, it has projective
dimension one. Consequently, the left $A$-module
$I$ is projective. Since $A$ is connected graded,
$I$ is free (of rank one). Thus $I= A x$ for some 
homogeneous element $x\in A$. By symmetry, $I=y A$
for some homogeneous element $y$. Then 
$Ax=yA$ implies that $x=_{\Bbbk^{\times}} y$. 
Thus $x$ is normal and the assertion follows.
\end{proof}

We use Lemma \ref{yylem3.14} to make the following definitions
in the case that $A$ has global dimension two.

\begin{definition}
\label{yydef3.15}
Assume Hypotheses \ref{xxhyp0.1}. Further assume
that $A$ has global dimension two and that the
radical ideal of the $H$-action ${\mathfrak r}_{A,H}$ is nonzero.
\begin{enumerate}
\item[(1)]
Any element that generates the principal ideal 
$\overline{{\mathfrak r}_{A,H}}$ in $A$ is called
a {\it principal radical} of the $H$-action on
$A$, and is denoted by $\widetilde{\mathfrak r}_{A,H}$.
\item[(2)]
Any element that generates the principal ideal 
$\overline{\Delta_{A,H}}$ in $R$ is called
a {\it principal dis-radical} of the $H$-action on
$A$, and is denoted by $\widetilde{\Delta}_{A,H}$.
\end{enumerate}
\end{definition}

Note that the principal radical 
$\widetilde{\mathfrak r}_{A,H}$ is always defined for any 
Hopf algebra $H$ acting on an AS regular algebra $A$ of 
global dimension 2, while the principal dis-radical 
$\widetilde{\Delta}_{A,H}$ is defined only when, in addition, 
$H$ is a reflection Hopf algebra.

\section{Examples}
\label{xxsec4}

When a Hopf algebra $H$ acts on a noetherian AS regular algebra $A$,
there is a list of important invariants that can be studied. Starting
from $A$, we can consider the following data:
\begin{enumerate}
\item[($\bullet$)]
the Nakayama automorphism of $A$, denoted by $\mu$ [Definition \ref{xxdef1.2}].
\item[($\bullet$)]
the AS index of $A$, denoted by ${\mathfrak l}$  [Definition \ref{xxdef1.1}].
\item[($\bullet$)]
the twisted superpotential associated to $A$ 
\cite[Definition 1]{DV} or \cite[p.1502]{BSW}.
\end{enumerate}
For $H$, since we assume that $H$ is semisimple, it is Calabi-Yau
with trivial Nakayama automorphism. When $H$ acts on $A$, 
we can consider:
\begin{enumerate}
\item[($\bullet$)]
the pertinency $\p(A,H)$ [Definition \ref{xxdef2.8}(3)].
\item[($\bullet$)]
the pertinency ideal ${\mathcal P}(A,H)$ [Definition \ref{xxdef2.8}(1)].
\item[($\bullet$)]
the radical ideal ${\mathfrak r}_{A,H}$ [Definition \ref{xxdef2.8}(1)].
In global dimension two case, we can ask for the principal radical
$\widetilde{\mathfrak r}_{A,H}$ [Definition \ref{yydef3.15}(1)].
\item[($\bullet$)]
$H$-dis-radical ideal $\Delta_{A,H}$ [Definition \ref{xxdef3.1}(4)]. 
In global dimension two case, we can ask for the principal dis-radical
$\widetilde{\Delta}_{A,H}$ [Definition \ref{yydef3.15}(2)].
\item[($\bullet$)] the
homological determinant $\hdet$ of the $H$-action on $A$ \cite[Definition 3.3]{KKZ3}.
\item[($\bullet$)] the
fusion rules for $H$, or the McKay quiver for representations
of $H$.
\end{enumerate}
When the fixed subring $A^H$ is AS Gorenstein (or AS regular), we can 
further consider:
\begin{enumerate}
\item[($\bullet$)] the 
Jacobian ${\mathsf j}_{A,H}$ [Definition \ref{xxdef2.1}(1)].
\item[($\bullet$)] the 
reflection arrangement ${\mathsf a}_{A,H}$ [Definition \ref{xxdef2.1}(2)].
\item[($\bullet$)]
${\mathfrak R}^l({\mathsf j}_{A,H})$ and ${\mathfrak R}^l({\mathsf a}_{A,H})$,
see \eqref{E0.9.1}.
\item[($\bullet$)] the 
discriminant $\delta_{A,H}$ [Definition \ref{xxdef3.1}(3)]. 
\end{enumerate}
There are several algebras associated to $(A,H)$: the fixed subring 
$A^H$, the covariant ring $A^{cov\; H}$, the 
$G$-component $A_{G}$, $A/({\mathsf a}_{A,H})$ if ${\mathsf a}_{A,H}$ 
is normal, $R/(\delta_{A,H})$ when $\delta_{A,H}$ is defined. 
If any of the these algebras is AS Gorenstein, we can compute the 
corresponding data in the first two $\bullet$s. 

First we compute the Jacobian
when $A = \Bbbk_{-1}[x,y]$ and $H$ is a group algebra $\Bbbk G$
for some finite group $G$. Let us recall some facts from 
\cite{KKZ2}. We consider two different kinds of automorphisms
of $\Bbbk_{-1}[x,y]$. The first is of the form
\begin{equation}
\label{E4.0.1}\tag{E4.0.1}
\sigma_a: x\mapsto ax, y\mapsto  y, \qquad 
{\text{or}} \qquad \tau_b:  x\mapsto x, y\mapsto ay.
\end{equation}
and the
second one is of the form
\begin{equation}
\label{E4.0.2}\tag{E4.0.2}
\tau_{1,2,\lambda}: x\mapsto \lambda y, y\mapsto -\lambda^{-1} x.
\end{equation} 
Let $\alpha$ and $\beta$ be two positive integers such that 
$\beta$ is divisible by both $2$ and $\alpha$. Let 
$M(2, \alpha,\beta)$ be the subgroup of $\Aut_{gr}(\Bbbk_{-1}[x,y])$
generated by 
$$\{\sigma_{a}\mid a^\alpha=1\} 
\cup  \{\tau_{1,2,\lambda}\mid \lambda^\beta=1\}$$
(see 
\cite{KKZ2} in discussion before \cite[Lemma 5.3]{KKZ2}).
By \cite[Lemma 5.3]{KKZ2}, if $G$ is not generated only by a single
$\sigma_a$ or $\tau_{1,2,a}$ in \eqref{E4.0.1}-\eqref{E4.0.2} and 
$\Bbbk_{-1}[x,y]^G$ is AS regular, then $G\cong M(2, \alpha,\beta)$.
As one example, the groups $M(2, 1, 2 \ell )$ are the binary 
dihedral groups of order $4\ell$ generated by $\tau_{1,2,1}$ 
and $\tau_{1,2,\lambda}$ for $\lambda$ a primitive $2\ell$th 
root of unity, i.e. the representation generated by the two mystic reflections:
$$g_1 = \begin{pmatrix}
0 & 1\\
-1 & 0
\end{pmatrix}
\text{ and } g_2=
\begin{pmatrix}
0 & \lambda\\
-\lambda^{-1} & 0
\end{pmatrix}.$$

\begin{lemma}
\label{xxlem4.1} 
Suppose that $A=\Bbbk_{q}[x,y]$ where $1\neq q\in \Bbbk^{\times}$ 
and that $H =\Bbbk G$ for a finite group $G$. Assume 
Hypotheses \ref{xxhyp0.1}. Then one of the following holds.
\begin{enumerate}
\item[(1)]
$G=\langle \sigma \rangle \times \langle \tau \rangle
\cong C_n \times C_m$ where $\sigma$ and $\tau$ are 
of the form given in \eqref{E4.0.1} and of order $n$ 
and $m$ respectively. In this case ${\mathsf j}_{A,H}
=_{\Bbbk^{\times}} x^{n-1} y^{m-1}$, ${\mathsf a}_{A,H}
=_{\Bbbk^{\times}} x y$ and 
$${\mathfrak R}^l({\mathsf j}_{A,H})
={\mathfrak R}^l({\mathsf a}_{A,H})
={\mathfrak R}^r({\mathsf j}_{A,H})
={\mathfrak R}^r({\mathsf a}_{A,H})
=\{\Bbbk x, \Bbbk y\}.$$
\item[(2)]
$q=-1$ and $G=M(2, \alpha,\beta)$ for $\alpha\geq 2$.  Then
$${\mathsf a}_{A,H}=xy(x^{\beta}-y^{\beta})\quad 
{\text{and}}\quad
{\mathsf j}_{A,H}=x^{\alpha -1}y^{\alpha-1}(x^{\beta}-y^{\beta}).
$$ 
Further,
$${\mathfrak R}^l({\mathsf j}_{A,H})=
{\mathfrak R}^r({\mathsf j}_{A,H})=
{\mathfrak R}^l({\mathsf a}_{A,H})=
{\mathfrak R}^r({\mathsf a}_{A,H})=
\{\Bbbk x,\Bbbk y\} \cup \{\Bbbk (x+\xi y)\mid \xi^{\beta}=1\}.$$
\item[(3)]
$q=-1$ and $G=M(2, 1,\beta)$ {\rm{(}}for $\alpha=1${\rm{)}}. 
Then
$${\mathsf a}_{A,H}=
{\mathsf j}_{A,H}=(x^{\beta}-y^{\beta}).
$$ 
Further,
$${\mathfrak R}^l({\mathsf j}_{A,H})=
{\mathfrak R}^r({\mathsf j}_{A,H})=
{\mathfrak R}^l({\mathsf a}_{A,H})=
{\mathfrak R}^r({\mathsf a}_{A,H})=
\{\Bbbk (x+\xi y)\mid \xi^{\beta}=1\}.$$
\end{enumerate}
\end{lemma}

\begin{proof} 
By \cite[Theorem 1.1]{KKZ2}, $G$ is generated by 
quasi-reflections in the sense of \cite[p. 131]{KKZ2}. 
When $q\neq \pm 1$, $\Aut_{gr}(A)=(\Bbbk^{\times})^2$
and every quasi-reflection is a reflection in the 
sense of \cite[Definition 2.3(1)]{KKZ2}, namely, of the 
form in \eqref{E4.0.1}. One can check easily from this 
observation that $G\cong C_n \times C_m$. The statements
in part (1) are easy to check now.

When $q=-1$, one extra possibility is that $G$ is generated 
by mystic reflections in the sense of \cite[Definition 2.3(1)]{KKZ2}. 
In this case, by \cite[Lemma 5.3]{KKZ2}, $G$ is the group
$M(2,\alpha,\beta)$, and $A^G$ is the
commutative polynomial ring $\Bbbk [x^\alpha y^\alpha, x^{\beta}+
y^{\beta}]$ \cite[Proposition 5.4]{KKZ2}.  

(2) When $\alpha\geq 2$ one can check directly that nonzero 
elements of the minimal degree in $A_{\hdet}$ and $A_{\hdet^{-1}}$ are
$${\mathsf a}_{A,\Bbbk G}=xy(x^{\beta}-y^{\beta})\quad 
{\text{and}}\quad
{\mathsf j}_{A,\Bbbk G}=x^{\alpha -1}y^{\alpha-1}(x^{\beta}-y^{\beta})
$$
respectively. From this we obtain, after an easy calculation, that
$${\mathfrak R}^l({\mathsf j}_{A,\Bbbk G})=
{\mathfrak R}^r({\mathsf j}_{A,\Bbbk G})=
{\mathfrak R}^l({\mathsf a}_{A,\Bbbk G})=
{\mathfrak R}^r({\mathsf a}_{A,\Bbbk G})=
\{\Bbbk x,\Bbbk y\} \cup \{\Bbbk (x+\xi y)\mid \xi^{\beta}=1\}.$$
The assertion follows.

(3) When $\alpha=1$, the computation is similar to the one in part (2).
\end{proof}

By Lemma \ref{xxlem4.1}(2,3), $\Bbbk M(2,\alpha,\beta)$ is a true 
reflection Hopf algebra acting on $\Bbbk_{-1}[x,y]$ if and only if 
$\alpha=1$ or $2$. When $\alpha=1$, $M(2,1,\beta)$ is isomorphic to 
a binary dihedral group. Note in this case that the number of mystic 
reflections is the degree of ${\mathsf j}_{A,\Bbbk G}$, also 
equals to $|{\mathfrak R}^l({\mathsf j}_{A,\Bbbk G})|$. 
Further that the Jacobian (and hence the reflection arrangement) is 
central, but $A^G$ is not central in $A$.

For the rest of this section we give an example where $H$ is neither 
commutative or cocommutative. This example is the smallest possible  
in terms of dimensions, $H$ having $\Bbbk$-dimension 8 and
$A$ having global dimension 2. Even in this ``small'' example,
computations are still quite complicated, unfortunately. To save 
some space, some non-essential details are omitted, especially 
towards the end of the example. Some additional information concerning
this example is given in \cite{FKMW1} and \cite[Example 7.4]{KKZ3}.

\begin{example}
\label{xxex4.2}
Assume that ${\rm{char}}\; \Bbbk=0$. 
Let $H$ be the Kac-Palyutkin Hopf algebra $H_8$. By \cite[p.341]{BN}, 
$H$ is self-dual and it has no nontrivial dual cocycle twist in the 
sense of \cite{Ma2,Mo2}. Recall that $H$ is generated by $x, y, z$ and 
subject to the following relations:
$$x^2 = y^2 =1, \;\; xy=yx,\;\; zx=yz,$$
$$ zy=xz,\;\; z^2= \frac{1}{2}(1+x+y-xy). $$
The comultiplication of $H$ is determined by:
$$\begin{aligned}
\Delta(x) & = x\otimes x,\\ 
\Delta(y) & = y\otimes y,\\
\Delta(z) & = \frac{1}{2}(1\otimes 1 + 1\otimes x 
+ y\otimes 1 - y\otimes x)(z\otimes z).
\end{aligned}
$$
The group of grouplike elements in $H$ is 
$G(H) = \{1,x,y,xy\}$, the 
Klein four group. Let $\Bbbk_{i}[u,v]$ be the skew polynomial algebra 
generated by $u,v$ and subject to the relation
\begin{equation}
\label{E4.2.1}\tag{E4.2.1}
vu =i uv
\end{equation}
where $i^2=-1$. By \cite[Example 5.5]{RRZ2}, the Nakayama automorphism 
of $A$ is determined by
$$\mu: u \mapsto -i u, \quad v\mapsto i v$$
and, in dimension two, the twisted superpotential is trivially
the single relation, namely, 
$$\omega= vu-i uv.$$
By \cite[Example 7.4]{KKZ3}, $H$ acts on $A:=\Bbbk_{i}[u,v]$
inner-faithfully with commutative (but not central) regular fixed subring 
$A^H=\Bbbk[u^2+v^2,u^2v^2]$. Thus Hypotheses \ref{xxhyp0.1} (and hence
Hypotheses \ref{xxhyp3.2}) holds. 

It is easy to check that $AR_{\geq 1}\neq R_{\geq 1}A$. So the 
$H$-action on $A$ is not tepid, see Definition \ref{xxdef1.11}(4).
It is routine to check that the covariant algebra
$A^{cov\; H}:=A/(R_{\geq 1})$ is isomorphic to 
$A/(\Bbbk(u^2+v^2)\oplus A_{\geq 3})$, 
which has Hilbert series $1+2t+2t^2$. As a consequence, 
$A^{cov\; H}$ is not Frobenius, which is different from the 
classical (commutative) case and the case of the dual reflection 
groups in \cite[Theorem 0.4]{KKZ4}.

Recall from \cite[Example 7.4]{KKZ3} 
that there is a unique two-dimensional $H$-representation 
$V = \Bbbk u \oplus \Bbbk v$ given by the assignment:
$$x \rightarrow \begin{pmatrix}-1 & 0\\ 0 & 1 \end{pmatrix}, \;\; 
y \rightarrow \begin{pmatrix}1 & 0\\ 0 & -1\end{pmatrix}, \;\; 
z \rightarrow \begin{pmatrix}0 & 1 \\ 1 & 0\end{pmatrix},$$
which uniquely determines the $H$-action on $A=\Bbbk_i[u,v]$.

Our first goal is to calculate the Jacobian, the reflection arrangement 
and the discriminant of this $H$-action on $A$. Note that $x+y, xy,$ and 
$z^2$ are all central in $H$. Consider the central idempotents in $H$:
$$f_1 = (1 + x + y + xy)/4, \quad \text{and} \quad f_2= (1 -x-y+xy)/4.$$
It is easy to check that $f_1z^2 = z^2 f_1 = f_1$ and $f_2 z^2 = z^2 f_2 = -f_2$.
In addition we have the following two idempotents in $H$ that are not central:
$$f_3 = (1 - x + y - xy)/4, \quad \text{and} \quad f_4= (1 +x-y-xy)/4.$$
These idempotents satisfy  $f_if_j = 0$ for all $i\neq j$.
Using the above information, we define the following central idempotents of 
$H$, that correspond to the group of  grouplike elements $G(K)(=\{1, g, g', gg'\})$,
where $K$ is the dual Hopf algebra of $H$,
$$\begin{aligned}
p_{1} &= \inth = (f_1 + zf_1)/2 = (1 +x+y+xy + z +xz + yz + xyz)/8\\
p_{g} &= (f_1 - zf_1)/2 = (1 +x+y+xy - z -xz - yz - xyz)/8\\
p_{g'} &=  (f_2+ izf_2)/2 = (1 -x-y+xy +i z -ixz -i yz  +ixyz)/8\\
p_{gg'} &=  (f_2- izf_2)/2 = (1 -x-y+xy -i z +ixz + i yz  -ixyz)/8.
\end{aligned}
$$
Using the fact that $zu^2 = v^2, z v^2 = u^2, z (uv) = -iuv, z(u^3v) = iuv^3$, etc., 
we obtain
$$\begin{aligned}
p_{1} A &= A^H = R = \Bbbk[u^2+v^2,u^2v^2]\\
p_{g} A &= (u^2-v^2)R\\
p_{g'} A &= (uv)R\\
p_{gg'} A &= (u^3v+uv^3) R = (uv(u^2-v^2))R.
\end{aligned}
$$
As a consequence, we have the decomposition of $A_G$ into graded pieces 
(as in Lemma \ref{xxlem1.5}(3))
$$A_G = p_{1}A \oplus p_{g} A \oplus p_{g'} A \oplus p_{gg'}A=A^{(2)},$$ 
where $A^{(2)}$ is the second Veronese subring of $A$. It is clear that
$(u^2+v^2) u=u(u^2-v^2)$ which is not in $AR_{\geq 1}$. Hence 
$R_{\geq 1}A\neq AR_{\geq 1}$, and consequently, the $H$-action on $A$ is not
tepid in the sense of Definition \ref{xxdef1.11}(4). By an easy 
calculation, 
$$\xi(t)=h_A(t) (h_{R}(t))^{-1}=(1+t)(1+t+t^2+t^3)$$
which has degree 4. It follows from Corollary \ref{xxcor2.5}(2) then 
$\deg {\mathsf j}_{A,H}=4$. Hence,
\begin{equation}
\label{E4.2.2}\tag{E4.2.2}
{\mathsf j}_{A,H}=_{\Bbbk^{\times}} uv(u^2-v^2)
\end{equation}
and 
\begin{equation}
\label{E4.2.3}\tag{E4.2.3}
\hdet^{-1}=gg'.
\end{equation}
Since $\hdet^2=1$, we obtain that
\begin{equation}
\label{E4.2.4}\tag{E4.2.4}
{\mathsf a}_{A,H}={\mathsf j}_{A,H}=_{\Bbbk^{\times}} uv(u^2-v^2)
\end{equation}
and
\begin{equation}
\label{E4.2.5}\tag{E4.2.5}
\delta_{A,H}=_{\Bbbk^{\times}} u^2 v^2 (u^2-v^2)^2
=u^2v^2[(u^2+v^2)^2-4u^2v^2] \in R.
\end{equation}
As a consequence of \eqref{E4.2.4}, $H$ is a true reflection Hopf 
algebra. Using the fact that
$$u^2-v^2=(u+e^{\frac{3}{8} (2\pi i)}v)((u+e^{\frac{1}{8} (2\pi i)}v)
=(u+e^{\frac{7}{8} (2\pi i)}v)((u+e^{\frac{5}{8} (2\pi i)}v)$$
and that $u$ and $v$ are normal, we can calculate
\begin{equation}
\label{E4.2.6}\tag{E4.2.6}
{\mathfrak R}^l({\mathsf a}_{A,H})={\mathfrak R}^l({\mathsf j}_{A,H})
=\{ \Bbbk u, \Bbbk v, \Bbbk (u+e^{\frac{3}{8} (2\pi i)}v),
\Bbbk (u+e^{\frac{7}{8} (2\pi i)}v)\}
\end{equation}
and
\begin{equation}
\label{E4.2.7}\tag{E4.2.7}
{\mathfrak R}^r({\mathsf a}_{A,H})={\mathfrak R}^r({\mathsf j}_{A,H})
=\{ \Bbbk u, \Bbbk v, \Bbbk (u+e^{\frac{1}{8} (2\pi i)}v),
\Bbbk (u+e^{\frac{5}{8} (2\pi i)}v)\}.
\end{equation}

Since ${\mathsf j}_{A,H}$ is not normal (easy to check), $A{\mathsf j}_{A,H}$
is not a 2-sided ideal. So ${\mathfrak r}_{A,H}\neq A{\mathsf j}_{A,H}$.
By Lemma \ref{yylem3.13} (after verifying the hypotheses in Lemma \ref{yylem3.13}),
${\mathfrak r}_{A,H}$ is a subspace of $A{\mathsf j}_{A,H}$.

Our second goal is to calculate the radical ideal of this $H$-action.
Let 
$$E = (1-xy)/2 = 1- (p_{1}+p_{g}+p_{g'}+p_{gg'}),$$
a central idempotent of $H$ so 
$$H = \Bbbk p_{1} \oplus \Bbbk p_{g} \oplus  \Bbbk p_{g'} 
\oplus \Bbbk p_{gg'} \oplus E H.$$
We have the relations:
$$E f_3 = f_3, \;\; E f_4 = f_4,\;\; f_3f_4 
= f_4 f_3= 0, \;\; z f_3 = f_4 z,\;\; z f_4 = f_3 z.$$
Further
$$f_3 z^2 = \frac{1}{8} (1-x + y -xy) (1+x + y -xy) 
= \frac{1}{8}(2 - 2x + 2y -2 xy) = f_3$$
and similarly
$$f_4 z^2 = \frac{1}{8} (1+x - y -xy) (1+x + y -xy) 
= \frac{1}{8}(2 + 2x - 2y -2 xy) = f_4.$$

Hence let $m_{12} = f_3zf_4= f_3z = z f_4$ and $m_{21} = f_4z f_3 = f_4z = zf_3$.
Then 
$$\begin{aligned}
m_{12}m_{21} &= (f_3z)(z f_3)= f_3 z^2 f_3 = f_3 f_3 = f_3,\\
m_{21}m_{12} &= (f_4z)(z f_4) = f_4,
\end{aligned}
$$
and 
$$\begin{aligned}
m_{12}^2 &= (f_3zf_4)(f_3zf_4) = 0,\\
m_{21}^2 &= (f_4zf_3)(f_4zf_3) =0.
\end{aligned}
$$
So the subspace $EH$ is isomorphic to $2\times 2$-matrix, and for convenience, 
we write 
$$EH  \cong \begin{pmatrix} f_3 & m_{12}\\m_{21} & f_4 \end{pmatrix}
=\begin{pmatrix} m_{11} & m_{12}\\m_{21} & m_{22} \end{pmatrix}.$$

Next we find $x_{i,j}, y_{i,j}$ so that
$$\Delta(\smallint) = \Delta(p_{1}) 
= p_{1} \otimes p_{1} +  p_{g} \otimes p_{g} 
+  p_{g'} \otimes p_{g'} +  p_{gg'} \otimes p_{gg'}$$
$$ + \sum_{1 \leq i,j \leq 2} m_{ij}\otimes x_{ij} 
+ \sum_{1 \leq i,j \leq 2} y_{ij} \otimes m_{ij}.$$
First compute 
$$
\begin{aligned}
\Delta(\smallint) &= \Delta((1+x+y+xy) (1 +z)/8) = \Delta(1+x+y+xy) \Delta(1+z)/8\\
&= \frac{1}{8}(1 \otimes 1 + x \otimes x + y \otimes y + xy \otimes xy)(1 \otimes 1 \\
&\qquad\qquad+ \frac{z \otimes z  + z \otimes xz + yz \otimes z - yz \otimes xz}{2})\\
&= \frac{(1 \otimes 1 + x \otimes x + y \otimes y + xy \otimes xy)}{8}\\
&\qquad\qquad + \frac{(z \otimes z + xz \otimes xz + yz \otimes yz + xyz \otimes xyz)}{16}\\
&\qquad\qquad + \frac{(z \otimes xz + xz \otimes z + yz \otimes xyz + xyz \otimes yz)}{16}\\
&\qquad\qquad + \frac{(yz \otimes z + xyz \otimes xz + z \otimes yz + xz \otimes xyz)}{16}\\
&\qquad\qquad - \frac{(yz \otimes xz + xyz \otimes z + z \otimes xyz + xz \otimes yz)}{16}.
\end{aligned}
$$
After some tedious computation, we obtain that
$$\begin{aligned}
\Delta(\inth)  &= \Delta(p_{1}) = p_{1} \otimes p_{1} +  p_{g} \otimes p_{g} 
+  p_{g'} \otimes p_{g'} +  p_{gg'} \otimes p_{gg'}\\
&\qquad +  ( (f_3 \otimes f_3) +  (f_4 \otimes f_4) + (m_{12} \otimes m_{12}) 
+  (m_{21} \otimes m_{21}))/2.
\end{aligned}
$$

Let $L$ be the left ideal of $A$ generated by elements $w$ satisfying
\begin{align}
\label{E4.2.8}\tag{E4.2.8}
w&=\sum_{i} b_i (f_3\cdot a_i)+\sum_j d_j (m_{12}\cdot c_j)\\
\label{E4.2.9}\tag{E4.2.9}
0&=\sum_{i} b_i (m_{21}\cdot a_i)+\sum_j d_j (f_4 \cdot c_j)
\end{align}
for some $a_i, b_i, c_j, d_j$ in $A$. 
Let $L'$ be the left ideal of $A$ generated by 
elements $w$ satisfying
\begin{align}
\label{E4.2.10}\tag{E4.2.10}
w&=\sum_{i} b_i (f_4\cdot a_i)+\sum_j d_j (m_{21}\cdot c_j)\\
\label{E4.2.11}\tag{E4.2.11}
0&=\sum_{i} b_i (m_{12}\cdot a_i)+\sum_j d_j (f_3 \cdot c_j)
\end{align}
for some $a_i, b_i, c_j, d_j$ in $A$. 

It follows from the definition of the radical ideal 
[Definition \ref{xxdef2.8}(2)] we have

\begin{lemma}
\label{xxlem4.3}
Retain the above notation. The radical ideal is
$$({\mathfrak r}_{H,A})= A{\mathsf j}_{A,H}\cap L \cap L'.$$
\end{lemma}

\begin{proof} The main idea here is to do finer computations
than ones in the proof of Lemma \ref{yylem3.13}. To save space,
details are omitted.
\end{proof}

As a consequence, one can calculate the radical ideal in this 
example:
\begin{equation}
\label{E4.2.12}\tag{E4.2.12}
{\mathfrak r}_{A,H}=\overline{{\mathfrak r}_{A,H}}=({\tilde{\mathfrak r}}_{A,H})
\end{equation}
where
\begin{equation}
\label{E4.2.13}\tag{E4.2.13}
{\tilde{\mathfrak r}}_{A,H}=_{\Bbbk^{\times}}uv(u^4-v^4)=_{\Bbbk^{\times}}
{\mathsf j}_{A,H} (u^2+v^2).
\end{equation}
Further, 
\begin{equation}
\label{E4.2.14}\tag{E4.2.14}
\Delta_{A,H}=\overline{\Delta_{A,H}}=({\mathsf j}_{A,H}^2 \; (u^2+v^2))
=(\delta_{A,H} \; (u^2+v^2)).
\end{equation}
This is the end of the example.
\end{example}

Complex reflection groups are important in many areas of current research, 
for example in defining rational Cherednik algebras.  In this paper we 
have presented generalizations of the various invariants that are used 
in studying complex reflection groups, their geometry, and their actions 
on polynomial rings (see for example \cite{BFI}).  The tools developed 
here, the Jacobian, the reflection arrangement and the discriminant, as 
well as the pertinency ideal, the radical of the $H$-action, the 
homological determinant, and the Nakayama automorphism should further 
the understanding of Hopf actions on AS regular algebras. 
If there is ever a version of rational Cherednik algebras for Artin-Schelter 
regular algebras, then one should understand better reflection Hopf 
algebras, and whence, the invariants introduced in this paper.


\begin{thebibliography}{10} 


\bibitem[AP]{AP} 
J. Alev and P. Polo, 
A rigidity theorem for finite group actions on enveloping algebras of
semisimple Lie algebras, 
Adv. Math. {\bf 111} (1995), no. 2, 208--226. 


\bibitem[Ar]{Ar}
V. Artamonov, 
Actions of Pointed Hopf Algebras on Quantum Torus, 
Ann. Univ. Ferrara - Sez. VII - Sc. Mat. Vol. LI (2005), 29--60.

\bibitem[AC]{AC}
V.A. Artamonov and I.A. Chubarov, 
Properties of some semisimple Hopf algebras, 
Algebras, representations and applications, 23--36, 
Contemp. Math., {\bf 483}, Amer. Math. Soc., Providence, RI, 2009. 



\bibitem[AS]{AS}
 M. Artin and W. F. Schelter,  
Graded algebras of global dimension $3$, 
Adv. Math. {\bf 66} (1987), no. 2,
171--216.


\bibitem[AZ]{AZ}
M. Artin and J. J. Zhang, Noncommutative projective
schemes, Adv. Math. {\bf 109} (1994), 228-287.


\bibitem[BHZ1]{BHZ1}
Y.-H. Bao, J.-W. He and J.J. Zhang,
Pertinency of Hopf actions and quotient categories 
of Cohen-Macaulay algebras, 
J. Noncommut. Geom. {\bf 13} (2019), no. 2, 667--710.


\bibitem[BHZ2]{BHZ2}
Y.-H. Bao, J.-W. He and J.J. Zhang,
Noncommutative Auslander theorem, 
Trans. Amer. Math. Soc. {\bf 370} (2018), no. 12, 8613--8638. 


\bibitem[BZ]{BZ}
J. Bell and J. J. Zhang,
Zariski cancellation problem for noncommutative algebras,
Selecta Math. (N.S.) {\bf 23} (2017), no. 3, 1709--1737.

\bibitem[BN]{BN}
J. Bichon and S. Natale, 
Hopf algebra deformations of binary polyhedral groups,
Transform. Groups {\bf 16} (2011), no. 2, 339--374.

\bibitem[BSW]{BSW}
R. Bocklandt, T. Schedler and M. Wemyss, 
Superpotentials and higher order derivations, 
J. Pure Appl. Algebra {\bf 214} (2010), no. 9, 1501--1522.


\bibitem[BFI]{BFI}
R.-O. Buchweitz, E. Faber and C. Ingalls,
A McKay correspondence for reflection groups,
preprint (2017), arXiv:1709.04218.



\bibitem[CPWZ1]{CPWZ1}
S. Ceken, J. Palmieri, Y.-H. Wang and J.J. Zhang,
The discriminant controls automorphism groups of noncommutative
algebras, Adv. Math., {\bf 269} (2015), 551--584.


\bibitem[CPWZ2]{CPWZ2}
S. Ceken, J. Palmieri, Y.-H. Wang and J.J. Zhang,
The discriminant criterion and the
automorphism groups of quantized algebras,
Adv. Math. {\bf 286} (2016), 754--801.



\bibitem[CWZ]{CWZ}
K. Chan, C. Walton and J.J. Zhang, 
Hopf actions and Nakayama automorphisms, 
J. Algebra {\bf 409}  (2014), 26--53.


\bibitem[CKWZ1]{CKWZ1} 
K. Chan, E. Kirkman, C. Walton and J.J. Zhang, 
Quantum binary polyhedral groups
and their actions on quantum planes,  
J. Reine Angew. Math. {\bf 719} (2016), 211--252.

\bibitem[CKWZ2]{CKWZ2} 
K. Chan, E. Kirkman, C. Walton and J.J. Zhang, 
McKay Correspondence for semisimple
Hopf actions on regular graded algebras, I, 
J. Algebra {\bf 508} (2018), 512--538.

\bibitem[CKWZ3]{CKWZ3} 
K. Chan, E. Kirkman, C. Walton and J.J. Zhang, 
McKay Correspondence for semisimple
Hopf actions on regular graded algebras, II, 
J. Noncommut. Geom. {\bf 13} (2019), no. 1, 87--114.

\bibitem[CKZ1]{CKZ1}
J. Chen, E. Kirkman, J.J. Zhang,
Rigidity of down-up algebras with respect to finite group coactions, 
J. Pure Applied Algebra {\bf 221} (2017), 3089--3103.

\bibitem[CKZ2]{CKZ2}
J. Chen, E. Kirkman, J.J. Zhang,
Auslander's Theorem for group coactions on 
noetherian graded down-up algebras, preprint (2018),
arXiv:1801.09020, to appear Trans. Groups.


\bibitem[DV]{DV}
M. Dubois-Violette, 
Multilinear forms and graded algebras, 
J. Algebra {\bf 317} (1) (2007) 198--225.

\bibitem[FKMW1]{FKMW1}
L. Ferraro, E. Kirkman, W.F. Moore and R. Won,
Three infinite families of reflection Hopf algebras, preprint
(2018), arXiv:1810.12935.

\bibitem[FKMW2]{FKMW2}
L. Ferraro, E. Kirkman, W.F. Moore and R. Won,
Semisimple reflection Hopf algebras of dimension sixteen,
preprint (2019), arXiv:1907.06763.


\bibitem[GKMW]{GKMW}
J. Gaddis, E. Kirkman, W.F. Moore and R. Won, 
Auslander's Theorem for permutation actions on 
noncommutative algebras, 
Proc. Amer. Math. Soc. {\bf 147} (2019), no. 5, 1881--1896.

\bibitem[HS]{HS}
J. Hartmann and A.V. Shepler, 
Jacobians of reflection groups, 
Trans. Amer. Math. Soc. {\bf 360} (2008), no. 1, 123--133. 

\bibitem[HZ]{HZ} 
J.-W. He and Y. Zhang, 
Local cohomology associated to the radical of a group action 
on a noetherian algebra, 
Israel J. Math. {\bf 231} (2019), no. 1, 303--342. 


\bibitem[Jo]{Jo}
P. J{\o}rgensen, 
Non-commutative graded homological identities, 
J. London Math. Soc. (2) {\bf 57}  (1998),  no. 2, 336--350.

\bibitem[JZ]{JZ}
P. J{\o}rgensen and J.J. Zhang, Gourmet's guide to Gorensteinness,
Adv. Math. {\bf 151} (2000), no. 2, 313--345.


\bibitem[KP]{KP}  
G.I. Kac and V.G. Pulyutkin,
Finite ring groups,
Trudy Moskov. Mat. Obsc, {\bf 15} (1966), 224--261.


\bibitem[Ki]{Ki} 
E. Kirkman, Invariant theory of Artin-Schelter regular algebras: 
a survey, Recent developments in representation theory, 25--50, 
Contemp. Math., {\bf 673}, Amer. Math. Soc., Providence, RI, 2016.


\bibitem[KKZ1]{KKZ1} E. Kirkman, J. Kuzmanovich, and J.J. Zhang,
Rigidity of graded regular algebras,
Trans. Amer. Math. Soc. {\bf 360} (2008), no. 12, 6331--6369.

\bibitem[KKZ2]{KKZ2}
E. Kirkman, J. Kuzmanovich and J.J. Zhang, A
Shephard-Todd-Chevalley theorem for noncommutative regular
algebras, Algebr. Represent. Theory {\bf 13} (2010), no. 2, 127--158

\bibitem[KKZ3]{KKZ3}
E. Kirkman, J. Kuzmanovich and J.J. Zhang, Gorenstein subrings of invariants
under Hopf algebra actions, J. Algebra {\bf 322} (2009), no. 10, 3640--3669.


\bibitem[KKZ4]{KKZ4}
E.Kirkman, J. Kuzmanovich, and J.J. Zhang,
Nakayama automorphism and rigidity of dual reflection group coactions, 
J. Algebra {\bf 487} (2017), 60--92.




\bibitem[KL]{KL}
G. Krause and T. Lenagan, ``Growth of Algebras and
Gelfand-Kirillov Dimension'', revised edition, Graduate Studeis in
Mathematics, Vol. 22, AMS, Providence, 2000.

\bibitem[Ma1]{Ma1}
A. Masuoka, Private communications.

\bibitem[Ma2]{Ma2}
A. Masuoka, 
Cocyle deformations and Galois objects for some
cosemisimple Hopf algebras of finite dimension. 
In New trends in Hopf
algebra theorey (La Falda, 1999), volume {\bf 267} 
of Contemp. Math., pp. 195--214. Amer. Math. Soc., 
Providence, RI, 2000.


\bibitem[MR]{MR}
J. C. McConnell and J. C . Robson, ``Noncommutative Noetherian
Rings,'' Wiley, Chichester, 1987.

\bibitem[Mo1]{Mo1}
S. Montgomery,
``Hopf algebras and their actions on rings'',
CBMS Regional Conference Series in
Mathematics, {\bf 82},  Providence, RI, 1993.

\bibitem[Mo2]{Mo2} 
S. Montgomery,
Algebra properties invariant under twisting, 
In Hopf algebras in noncommutative
geometry and physics, volume {\bf 239} of Lecture Notes in 
Pure and Appl. Math., pp. 229--243. Dekker, New York, 2005



\bibitem[OT]{OT}
P. Orlik and H. Terao, 
``Arrangements of Hyperplanes'', 
Springer-Verlag, Berlin, 1992.

\bibitem[QWZ]{QWZ} 
X.-S. Qin, Y.-H. Wang and J.J. Zhang, 
Noncommutative quasi-resolutions, 
J. Algebra {\bf 536} (2019), 102--148.


\bibitem[RRZ1]{RRZ1}
Z. Reichstein, D. Rogalski and J.J. Zhang, 
Projectively simple rings, 
Adv. in Math, {\bf 203} (2006), 365--407.

\bibitem[RRZ2]{RRZ2}
M. Reyes, D. Rogalski and J.J. Zhang,
Skew Calabi-Yau algebras and homological identities, 
Adv. Math. {\bf 264} (2014), 308--354.

\bibitem[RRZ3]{RRZ3}
M. Reyes, D. Rogalski and J.J. Zhang,
Skew Calabi-Yau triangulated categories and Frobenius Ext-algebras,  
Trans. Amer. Math. Soc. {\bf 369} (2017), no. 1, 309--340.

\bibitem[Ro]{Ro}
J.J. Rotman, 
``An introduction to homological algebra'', 
Pure and Applied Mathematics,
{\bf 85}. Academic Press, Inc. New York-London, 1979. 


\bibitem[Sta]{Sta}
R.P. Stanley,
Relative invariants of finite
groups generated by pseudoreflections, J. Algebra {\bf 49} (1977),
134--148.


\bibitem[Ste]{Ste}
R. Steinberg, 
Invariants of finite reflection groups, Canad. J. Math. 
{\bf 12} (1960), 616--618.


\bibitem[StZ]{StZ}
D.R. Stephenson and J.J. Zhang, Growth of graded Noetherian rings,
Proc. Amer. Math. Soc. {\bf 125} (1997), no. 6, 1593--1605.


\bibitem[Te]{Te}
H. Terao, 
The Jacobians and the discriminants of finite reflection groups, 
Tohoku Math. J. (2) {\bf 41} (1989), no. 2, 237--247. 

\bibitem[VdB]{VdB}
M. Van den Bergh, 
Existence theorems for dualizing complexes over noncommutative 
graded and filtered rings, J. Algebra {\bf 195} (1997), 662--679.

\bibitem[YZ]{YZ}
A. Yekutieli and J.J. Zhang,
Rings with Auslander dualizing complexes,
J. Algebra {\bf 213} (1999), no. 1, 1--51.

\bibitem[Zh]{Zh}
J.J. Zhang, 
Connected graded Gorenstein algebras with enough normal elements, 
J. Algebra {\bf 189} (1997), no. 2, 390--405.

\end{thebibliography}
\end{document}